\documentclass[11pt]{article}

\usepackage[top=1.0in, bottom=1.0in, left=1.1in, right=1in]{geometry}

\usepackage{amssymb, stmaryrd, mathrsfs, amsthm, amsmath}

\usepackage{titletoc}

\usepackage{enumitem}  
\usepackage{verbatim}
\usepackage[nobysame, alphabetic, initials]{amsrefs}
\usepackage[english]{babel}

\usepackage{hyperref}

\hypersetup{colorlinks=false}

\makeatletter
\renewcommand{\BibLabel}{%
  \Hy@raisedlink{\hyper@anchorstart{cite.\CurrentBib}\relax\hyper@anchorend}%
  [\thebib]%
}
\makeatother

\usepackage[
open,
openlevel=2,
atend, numbered, 
]{bookmark}[2009/08/13]

\usepackage[arrow,matrix,curve]{xy}

\def\id{\textrm{id}}

\def\xto#1{\xrightarrow{#1}}

\newcommand{\nc}{\newcommand}

\newcommand{\myMO}[1]{{\fontshape{rm}{\textbf{#1}}}}

\DeclareMathOperator{\Mod}{\myMO{Mod}\hspace{+0.25ex}-\hspace{-0.25ex}}
\DeclareMathOperator{\Top}{\hspace{+0.25ex}-\hspace{-0.25ex}\myMO{Top}}

\newcommand*\cocolon{%
        \nobreak
        \mskip6mu plus1mu
        \mathpunct{}%
        \nonscript
        \mkern-\thinmuskip
        {:}%
        \mskip2mu
        \relax
}

\nc{\Ch}{\operatorname{Ch}}
\nc{\sA}{{\mathscr A}}
\nc{\wt}{\widetilde} \nc{\bl}{\bullet} \nc{\al}{\alpha}
\nc{\sg}{\sigma} \nc{\vf}{\varphi} \nc{\om}{\omega}
\nc{\ve}{\varepsilon} \nc{\ol}{\overline} \nc{\lb}{\lambda}
\nc{\Lb}{\Lambda} \nc{\Gm}{\Gamma} \nc{\cP}{{\mathscr P}}
\nc{\sB}{{\mathscr B}}
\nc{\ul}{\underline} \nc{\os}{\overset} \nc{\us}{\underset}
\nc{\pa}{\partial} \nc{\wh}{\widehat} \nc{\sbs}{\subset} \nc{\br}{\breve}
\nc{\lra}{\longrightarrow} \nc{\all}{\allowdisplaybreaks}
\nc{\Ker}{\operatorname{Ker}} \nc{\Img}{\operatorname{Im}}
\nc{\Kan}{\operatorname{Kan}} \nc{\Hom}{\operatorname{Hom}}
\nc{\Imm}{\operatorname{Im}}   \nc{\Ho}{\operatorname{Ho}}
\nc{\Ext}{\operatorname{Ext}}    \nc{\Cone}{\operatorname{Cone}}
\nc{\pr}{\operatorname{pr}} \nc{\cls}{\operatorname{cls}}
\nc{\cof}{\operatorname{cof}}
\nc{\sSet}{\operatorname{\textbf{sSet}}}
\nc{\Map}{\operatorname{Map}}
\nc{\Lan}{\operatorname{Lan}}
\nc{\incl}{\operatorname{incl}}
\nc{\Hocolim}{\operatorname{Hocolim}}
\nc{\colim}{\operatorname{colim}}
\nc{\Endd}{\operatorname{End}}
\nc{\const}{\operatorname{const}}
\nc{\inn}{\operatorname{in}}
\nc{\sk}{\operatorname{sk}}
\nc{\tr}{\operatorname{tr}}
\nc{\res}{\operatorname{res}}
\nc{\Res}{\operatorname{Res}}
\nc{\adj}{\operatorname{adj}}
\nc{\counit}{\operatorname{counit}}
\nc{\unit}{\operatorname{unit}}
\nc{\Fun}{\operatorname{Fun}}
\nc{\Fix}{\operatorname{Fix}}
\nc{\proj}{\operatorname{proj}}
\nc{\card}{\operatorname{card}}
\nc{\Tr}{\operatorname{Tr}}
\nc{\st}{\operatorname{st}}
\nc{\fib}{\operatorname{fib}}
\nc{\lv}{\operatorname{lv}}
\nc{\loc}{\operatorname{loc}}
\nc{\Or}{\operatorname{Or}}
\nc{\Aut}{\operatorname{Aut}}

\newcommand\even{\makebox(0,0){\text{even}}}

\nc{\rr}{\operatorname{\textbf{r}}}
\nc{\cop}{\operatorname{\textbf{l}}}

\nc{\Ev}{\operatorname{Ev}}

\nc{ \Sp}{\operatorname{Sp}}

\makeatletter
\@addtoreset{equation}{section}
\makeatother

\numberwithin{equation}{subsection}
\newtheorem{theo}[equation]{Theorem}
\newtheorem{lem}[equation]{Lemma}
\newtheorem{prop}[equation]{Proposition}
\newtheorem{coro}[equation]{Corollary}
\theoremstyle{definition}
\newtheorem{defi}[equation]{Definition}
\newtheorem{conv}[equation]{Convention}
\newtheorem{remk}[equation]{Remark}

\newtheorem{conj}[equation]{Conjecture}

\newtheorem*{quest}{Question}

\swapnumbers
\newtheorem{numbered paragraph}[equation]{}

\newenvironment{proof_ohne_pt}{{\noindent \bf Proof}}{{\hspace*{\fill}$\Box$\par\bigskip}}

\makeatletter \renewenvironment{proof}[1][\proofname] {\par\pushQED{\qed}\normalfont\topsep6\p@\@plus6\p@\relax\trivlist\item[\hskip\labelsep\bfseries#1\@addpunct{.}]\ignorespaces}{\popQED\endtrivlist\@endpefalse} \makeatother

\begin{document}

\def\N{{\mathbb N}}
\def\Z{{\mathbb Z}}
\def\S{{\mathbb S}}
\def\FF{{\mathbb F}}

\def\L{{\mathscr L}}
\def\D{{\mathscr D}}
\def\B{{\mathscr B}}
\def\BB{{\mathcal B}}
\def\Q{{\mathscr Q}}
\def\M{{\mathscr M}}
\def\D{{\mathscr D}}

\def\E{{\mathscr E}}
\def\K{{\mathscr K}}
\def\W{{\mathscr W}}
\def\F{{\mathscr F}}
\def\U{{\mathscr U}}

\def\T{{\mathscr T}}
\def\A{{\mathscr A}}
\def\C{{\mathscr C}}
\def\P{{\mathscr P}}
\def\V{{\mathscr V}}
\def\G{{\mathscr G}}
\def\X{{\mathscr X}}
\def\LL{{\mathbf L}}
\def\R{{\mathbf R}}
\def\RR{{\mathbb R}}
\def\ZZ{{\mathscr Z}}
\def\FF{{\mathbb F}}

\let\xto\xrightarrow

\title{\normalsize \bf RIGIDITY IN \\ EQUIVARIANT STABLE HOMOTOPY THEORY}

\author{\small IRAKLI PATCHKORIA}


\date{}

\maketitle

\begin{abstract}
For any finite group $G$, we show that the $2$-local $G$-equivariant stable homotopy category, indexed on a complete $G$-universe, has a unique equivariant model in the sense of Quillen model categories. This means that the suspension functor, homotopy cofiber sequences and the stable Burnside category determine all ``higher order structure'' of the $2$-local $G$-equivariant stable homotopy category, such as the equivariant homotopy types of function $G$-spaces. The theorem can be seen as an equivariant version of Schwede's rigidity theorem at the prime $2$.
\end{abstract}

\setcounter{section}{0}

\section{Introduction}

\setcounter{subsection}{1} 

One of the most difficult problems of algebraic topology is to calculate the stable homotopy groups of spheres. There has been extensive research in this direction establishing some remarkable results. A very important object used to do these kind of computations is the classical \emph{stable homotopy category} $SHC$. This category was first defined in \cite{Kanss} by Kan. Boardman in his thesis \cite{Boa64} constructed the (derived) smash product on $SHC$ whose monoids represent multiplicative cohomology theories. In \cite{BF78}, Bousfield and Friedlander introduced a stable model category $\Sp$ of spectra with $\Ho(\Sp)$ triangulated equivalent to $SHC$. The category $\Sp$ enjoys several nice point-set level properties. However, it does not possess a symmetric monoidal product that descends to Boardman's smash product on $SHC$. This initiated the search for new models for $SHC$ that possess symmetric monoidal products. In the 1990's several such models appeared: $S$-modules \cite{EKMM}, symmetric spectra \cite{HSS00}, simplicial (continuous) functors \cite{Lyd} and orthogonal spectra \cite{MMSS}. All these models turned out to be Quillen equivalent to $\Sp$ (and hence, to each other) and this naturally motivated the following

\begin{quest} How many models does $SHC$ admit up to Quillen equivalence? \end{quest}

In \cite{Sch07}, Schwede answered this question. He proved that the stable homotopy category is \emph{rigid}, i.e., if $\C$ is a stable model category with $\Ho(\C)$ triangulated equivalent to $SHC$, then the model categories $\C$ and $\Sp$ are Quillen equivalent. In other words, up to Quillen equivalence, there is a unique stable model category whose homotopy category is triangulated equivalent to the stable homotopy category. This theorem implies that all ``higher order structure'' of the stable homotopy theory, like, for example, homotopy types of function spaces, is determined by the suspension functor and the class of homotopy cofiber sequences.

Generally, when passing from a model category $\C$ to its homotopy category $\Ho(\C)$, one loses ``higher homotopical information'' such as homotopy types of mapping spaces in $\C$ or the algebraic $K$-theory of $\C$. In particular, the existence of a triangulated equivalence of homotopy categories does not necessarily imply that two given models are Quillen equivalent to each other. Here is an easy example of such a loss of information. Let $\Mod K(n)$ denote the model category of right modules over the $n$-th Morava $K$-theory $K(n)$ and let $\textbf{dg}\Mod \pi_*K(n)$ denote the model category of differential graded modules over the graded homotopy ring $\pi_*K(n)$. Then the homotopy categories $\Ho(\Mod K(n))$ and $\Ho(\textbf{dg}\Mod \pi_*K(n))$ are triangulated equivalent, whereas the model categories $\Mod K(n)$ and $\textbf{dg}\Mod \pi_*K(n)$ are not Quillen equivalent. The reason is that the homotopy types of function spaces in $\textbf{dg}\Mod \pi_*K(n)$ are products of Eilenberg-MacLane spaces which is not the case for $\Mod K(n)$ (see e.g. \cite[A.1.10]{Pat12}). 

Another important example which we would like to recall is due to Schlichting. It is easy to see that for any  prime $p$, the homotopy categories $\Ho(\Mod \Z/p^2)$ and $\Ho(\Mod \FF_p[t]/(t^2))$ are triangulated equivalent. In \cite{Schlich} Schlichting shows that the algebraic $K$-theories of the subcategories of compact objects of $\Mod \Z/p^2$ and $\Mod \FF_p[t]/(t^2)$ are different for $p \geq 5$. It then follows from \cite[Corollary 3.10]{DugShi04} that the model categories $\Mod \Z/p^2$ and $\Mod \FF_p[t]/(t^2)$ are not Quillen equivalent. Note that there is also a reinterpretation of this example in terms of differential graded alegbras \cite{DugShi09}.  

Initiated by Schwede's result, in recent years much research has been done on establishing essential uniqueness of models for certain homotopy categories. In \cite{Roi07}, Roitzheim shows that the $K_{(2)}$-local stable homotopy category has a unique model. For other theorems of this type see \cite{BR12} and \cite{Hut12}. 

The present work establishes a new uniqueness result. It proves an \emph{equivariant version} of Schwede's rigidity theorem at the prime $2$. Before formulating our main result, we would like to say a few words on equivariant stable homotopy theory. 

The $G$-\emph{equivariant stable homotopy category} (indexed on a complete $G$-universe), for any compact Lie group $G$, was introduced in the book \cite{LMS}. Roughly speaking, the objects of this category are $G$-spectra indexed on finite dimensional $G$-repre\-\-sentations. In this paper we will work with the stable model category $\Sp_G^O$ of $G$-equivariant orthogonal spectra indexed on a complete $G$-universe \cite{MM02}. The homotopy category of $\Sp_G^O$ is the $G$-equivariant stable homotopy category. The advantage of this model is that it possesses a symmetric monoidal product compatible with the model structure. As in the non-equivariant case, the $G$-equivariant stable homotopy category has some other monoidal models, like, for example, the category of orthogonal $G$-spectra equipped with the $\S$-model structure (flat model structure) \cite[Theorem 2.3.27]{Sto11}, the model category of $S_G$-modules \cite[IV.2]{MM02} and the model category of $G$-equivariant continuous functors \cite{Blum05}. For a finite group $G$, the model categories of $G$-equivariant topological symmetric spectra in the sense of \cite{Man04} and \cite{Haus} are also monoidal models for the $G$-equivariant stable homotopy category. Note that all these model categories are known to be $G \Top_*$-Quillen equivalent to each other (see \cite[IV.1.1]{MM02}, \cite[1.3]{Blum05}, \cite[2.3.31]{Sto11}, \cite{Man04} and \cite{Haus}).

\;

\;

Now we return to the actual content of this paper. Suppose $G$ is a finite group and $H$ a subgroup of $G$. For any $g \in G$, let ${}^g H$ denote the conjugate subgroup $gHg^{-1}$. Then the map 
$$g \colon \Sigma_{+}^{\infty}G/{}^g H \lra \Sigma_{+}^{\infty}G/H$$
in the homotopy category $\Ho(\Sp_G^O)$, given by $[x] \mapsto [xg]$ on the point-set level, is called the \emph{conjugation} map associated to $g$ and $H$. Further, if $K$ is another subgroup of $G$ such that $K \leq H$, then we have the \emph{restriction} map
$$\res_{K}^H \colon \Sigma_{+}^{\infty}G/{} K \lra \Sigma_{+}^{\infty}G/H$$
which is just the obvious projection on the point-set level. Moreover, there is also a map backwards, called the \emph{transfer} map
$$\tr_{K}^H \colon \Sigma_{+}^{\infty}G/{} H \lra \Sigma_{+}^{\infty}G/K,$$
given by the Pontryagin-Thom construction (see e.g. \cite[IV.3]{LMS} or \cite[II.8]{tom}). These morphisms generate the \emph{stable Burnside (orbit) category} which is the full preadditive subcategory of $\Ho(\Sp_G^O)$ with objects the stable orbits $\Sigma_{+}^{\infty}G/H$, $H \leq G$ \cite[V.9]{LMS} (see also \cite{Lew96}).  

Let $G$ be a finite group. We say that a model category $\C$ is a $G$-\emph{equivariant stable model category} if it is enriched, tensored and cotensored over the category $G \Top_*$ of pointed $G$-spaces in a compatible way (i.e., the pushout-product axiom holds) and if the adjunction
$$\xymatrix{ S^V \wedge -  \colon \C \ar@<0.5ex>[r] & \C \cocolon \Omega^V(-) \ar@<0.5ex>[l]}.$$
is a Quillen equivalence for any finite dimensional orthogonal $G$-representation $V$. 

All the models for the $G$-equivariant stable homotopy category mentioned above are $G$-equivariant stable model categories. Different kinds of equivariant spectra indexed on incomplete universes provide examples of $G \Top_*$-model categories which are not $G$-equivariant stable model categories but are stable as underlying model categories. 

\;

\;

Here is the main result of this paper:

\begin{theo} \label{damtkicebuli} Let $G$ be a finite group, $\C$ a cofibrantly generated, proper, $G$-equi\-\-variant stable model category, and let $\Sp_G^O,_{(2)}$ denote the $2$-localization of $ \Sp_G^O$. Suppose that 
$$\xymatrix{\Psi \colon \Ho( \Sp_G^O,_{(2)}) \ar[r]^-{\sim} & \Ho(\C)}$$
is an equivalence of triangulated categories such that 
$$\Psi(\Sigma_{+}^{\infty} G/H) \cong G/H_{+} \wedge^{\mathbf{L}} \Psi(\S),$$
for any $ H \leq G$. Suppose further that the latter isomorphisms are natural with respect to the restrictions, conjugations and transfers. Then there is a zigzag of $G \Top_*$-Quillen equivalences between $\C$ and $ \Sp_G^O,_{(2)}$.
\end{theo}

In fact, we strongly believe that the following integral version of Theorem \ref{damtkicebuli} should be true:

\begin{conj} \label{mainconj} Let $G$ be a finite group and let $\C$ be a cofibrantly generated, proper, $G$-equivariant stable model category. Suppose that 
$$\xymatrix{\Psi \colon \Ho( \Sp_G^O) \ar[r]^-{\sim} & \Ho(\C)}$$
is an equivalence of triangulated categories such that 
$$\Psi(\Sigma_{+}^{\infty} G/H) \cong G/H_{+} \wedge^{\mathbf{L}} \Psi(\S),$$
for any $ H \leq G$. Suppose further that the latter isomorphisms are natural with respect to the restrictions, conjugations and transfers. Then there is a zigzag of $G \Top_*$-Quillen equivalences between $\C$ and $ \Sp_G^O$.

\end{conj}

Note that if $G$ is trivial, then the statement of Conjecture \ref{mainconj} is true. This is Schwede's rigidity theorem \cite{Sch07}. (Or, more precisely, a special case of it, as the model category in Schwede's theorem need not be cofibrantly generated, topological or proper.) The solution of Conjecture \ref{mainconj} would in particular imply that all ``higher order structure'' of the $G$-equivariant stable homotopy theory such as, for example, equivariant homotopy types of function $G$-spaces, is determined by the suspension functor, the class of homotopy cofiber sequences and the basic $\pi_0$-information of $\Ho(\Sp_G^O)$, i.e., the stable Burnside (orbit) category.

The proof of Theorem \ref{damtkicebuli} is divided into two main parts: The first is categorical and the second is computational. The categorical part of the proof is mainly discussed in Section \ref{catinp} and essentially reduces the proof of Conjecture \ref{mainconj} to showing that a certain exact endofunctor 
$$F \colon \Ho(\Sp^O_G) \lra \Ho(\Sp^O_G)$$ 
is an equivalence of categories. The computational part shows that $2$-locally the endofunctor is indeed an equivalence of categories. The proof starts by generalizing Schwede's arguments from \cite{Sch01} to free (naive) $G$-spectra. From this point on, classical techniques of equivariant stable homotopy theory enter the proof. These include the double coset formula, Wirthm\"uller isomorphism, geometric fixed points, isotropy separation and the tom Dieck splitting. The central idea is to do induction on the order of subgroups and use the case of free $G$-spectra as the induction basis.

The only part of the proof of Theorem \ref{damtkicebuli} which uses that we are working $2$-locally is the part about free $G$-spectra in Section \ref{free}. The essential fact one needs here is that the self map $2 \cdot \id \colon M(2) \lra M(2)$ of the mod $2$ Moore spectrum is not zero in the stable homotopy category. For $p$ an odd prime, the map $p \cdot \id \colon M(p) \lra M(p)$ is equal to zero and this makes a big difference between the $2$-primary and odd primary cases. Observe that the nontriviality of $2 \cdot \id \colon M(2) \lra M(2)$ amounts to the fact that $M(2)$ does not possess an $A_2$-structure with respect to the canonical unit map $\S \lra M(2)$. In fact, for any prime $p$, the mod $p$ Moore spectrum $M(p)$ has an $A_{p-1}$-structure but does not admit an $A_p$-structure \cite{Ang08}. The obstruction for the latter is the element $\alpha_1 \in \pi_{2p-3}\S_{(p)}$. This is used by Schwede to obtain the integral rigidity result for the stable homotopy category in \cite{Sch07}. It seems to be rather nontrivial to generalize Schwede's obstruction theory arguments about coherent actions of Moore spaces \cite{Sch07} to the equivariant case.

This paper is organized as follows. Section \ref{prel} contains some basic facts about model categories and $G$-equivariant orthogonal spectra. We also review the level and stable model structures on the category of orthogonal $G$-spectra. In Section \ref{catinp} we discuss the categorical part of the proof. Here we introduce the category of orthogonal $G$-spectra $\Sp^O_G(\C)$ internal to an equivariant model category $\C$ and show that if $\C$ is stable in an equivariant sense and additionally satisfies certain technical conditions, then $\C$ and  $\Sp^O_G(\C)$ are Quillen equivalent. This allows us to reduce the proof of  Theorem \ref{damtkicebuli} to showing that a certain exact endofunctor $F$ of $\Ho(\Sp_G^O,_{(2)})$ is an equivalence of categories. In Section \ref{free} we show that $F$ becomes an equivalence when restricted to the full subcategory of free $G$-spectra. 

In Section \ref{redtrres} we prove that it is sufficient to check that the induced map
$$F \colon [\Sigma_{+}^{\infty} G/H, \Sigma_{+}^{\infty} G/H]_*^G \lra [F(\Sigma_{+}^{\infty} G/H), F(\Sigma_{+}^{\infty} G/H)]_*^G$$
is an isomorphism for any subgroup $H$ of $G$. This is then verified inductively in Section~\ref{prfmain}. The results of Section \ref{free} are used for the induction basis. The induction step uses geometric fixed points and a certain short exact sequence which we discuss in Section~\ref{geofixsec}.

\section*{Acknowledgements}

This paper is based on my PhD thesis at the University of Bonn. I would like to thank my advisor Stefan Schwede for suggesting the project, for introducing me to the beautiful world of equivariant stable homotopy theory, and for his help and encouragement throughout the way. I owe special thanks to Justin Noel for many hours of helpful mathematical discussions and for his help in simplifying some of the proofs in this paper. I am also indebted to Markus Hausmann, Kristian Moi and Karol Szumi\l o  for reading earlier drafts of the paper and for giving me valuable comments and suggestions. Further, I benefited from discussions with David Barnes, John Greenlees, Michael Hill, Michael Hopkins, Peter May, Lennart Meier, Constanze Roitzheim, Steffen Sagave, Brooke Shipley, Markus Szymik and many other people.

\;

This research was supported by the Deutsche Forschungsgemeinschaft Graduiertenkolleg 1150 ``Homotopy and Cohomology''. During the final revision of the paper the author was supported by the Danish National Research Foundation through the Centre for Symmetry and Deformation (DNRF92).

\section{Preliminaries} \label{prel}

\setcounter{subsection}{0} 

\subsection{Model categories} \label{modcatI}

A \emph{model category} is  a bicomplete category $\C$ equipped with three classes of morphisms called weak equivalences, fibrations and cofibrations, satisfying certain axioms. We will not list these axioms here. The point of this structure is that it allows one to ``do homotopy theory'' in $\C$. Good references for model categories include \cite{DS95}, \cite{Hov99} and \cite{Q67}.

The fundamental example of a model category is the category of topological spaces (\cite{Q67}, \cite[2.4.19]{Hov99}). Further important examples are the category of simplicial sets (\cite{Q67}, \cite[I.11.3]{GJ99}) and the category of chain complexes of modules over a ring \cite[2.3.11]{Hov99}.

For any model category $\C$, one has the associated \emph{homotopy category} $\Ho(\C)$ which is defined as the localization of $\C$ with respect to the class of weak equivalences (see e.g., \cite[1.2]{Hov99} or \cite{DS95}). The model structure guarantees that we do not face set theoretic problems when passing to localization, i.e., $\Ho(\C)$ has $\Hom$-sets.

A \emph{Quillen adjunction} between two model categories $\C$ and $\D$ is a pair of adjoint functors
$$\xymatrix{ F \colon \C \ar@<0.5ex>[r] & \D \cocolon E \ar@<0.5ex>[l]},$$
where the left adjoint $F$ preserves cofibrations and acyclic cofibrations (or, equivalently, $E$ preserves fibrations and acyclic fibrations). We refer to $F$ as a left Quillen functor and to $E$ as a right Quillen functor. Quillen's total derived functor theorem (see e.g., \cite{Q67} or \cite[II.8.7]{GJ99}) says that any such pair of adjoint functors induces an adjunction
$$\xymatrix{ \textbf{L}F \colon \Ho(\C) \ar@<0.5ex>[r] & \Ho(\D) \cocolon \textbf{R}E \ar@<0.5ex>[l]}.$$
The functor $\textbf{L}F$ is called the left derived functor of $F$ and $\textbf{R}E$ the right derived functor of $E$. If $\textbf{L}F$  is an equivalence of categories (or, equivalently, $\textbf{R}E$ is an equivalence), then the Quillen adjunction is called a \emph{Quillen equivalence}.

\noindent Next, recall (\cite{Q67}, \cite[6.1.1]{Hov99}) that the homotopy category $\Ho(\C)$ of a pointed model category $\C$ supports a \emph{suspension} functor
$$\Sigma \colon \Ho(\C) \longrightarrow \Ho(\C)$$
with a right adjoint \emph{loop} functor
$$\Omega \colon \Ho(\C) \longrightarrow \Ho(\C).$$
If the functors $\Sigma$ and $\Omega$ are inverse equivalences, then the pointed model category $\C$ is called a \emph{stable model category}. For any stable model category  $\C$, the homotopy category $\Ho(\C)$ is naturally triangulated \cite[7.1]{Hov99}. The suspension functor is the shift and the distinguished triangles come from the cofiber sequences. (We do not recall triangulated categories here and refer to \cite[Chapter IV]{GM} or \cite[10.2]{Wei} for the necessary background.) 

Examples of stable model categories are the model category of chain complexes and also various model categories of spectra ($S$-modules \cite{EKMM}, orthogonal spectra \cite{MMSS}, symmetric spectra \cite{HSS00}, sequential spectra \cite{BF78}).

\;

\;

For any stable model category $\C$ and objects $X, Y \in \C$, we will denote the abelian group of morphisms from $X$ to $Y$ in $\Ho(\C)$ by $[X,Y]^{\Ho(\C)}$.

\;

\;

Next, let us quickly review cofibrantly generated model categories. Here we mainly follow \cite[Section 2.1]{Hov99}. Let $I$ be a set of morphisms in an arbitrary cocomplete category. A \emph{relative $I$-cell complex} is a morphism that is a (possibly transfinite) composition of coproducts of pushouts of maps in $I$. A map is called \emph{$I$-injective} if it has the right lifting property with respect to $I$. An \emph{$I$-cofibration} is map that has the left lifting property with respect to $I$-injective maps. The class of $I$-cell complexes will be denoted by $I$-cell. Next, $I$-inj will stand for the class of $I$-injective maps and $I$-cof for the class of $I$-cofibrations. It is easy to see that $I$-cell $\subset$ $I$-cof. Finally, let us recall the notion of smallness. An object $K$ of a cocomplete category is small with respect to a given class $\D$ of morphisms if the representable functor associated to $K$ commutes with colimits of large enough transfinite sequences of morphisms from $\D$. See \cite[Definition 2.13]{Hov99} for more details.

\begin{defi}[{\cite[Definition 2.1.17]{Hov99}}] \label{cofgen} Let $\C$ be a model category. We say that $\C$ is \emph{cofibrantly generated}, if there are sets $I$ and $J$ of maps in $\C$ such that the following hold:

{\rm (i)} The domains of $I$ and $J$ are small relative to $I$-cell and $J$-cell, respectively.

{\rm (ii)} The class of fibrations is $J$-inj.

{\rm (iii)} The class of acyclic fibrations is $I$-inj.

\end{defi}

Here is a general result that will be used in this paper:

\begin{prop}[see e.g. {\cite[Theorem 2.1.19]{Hov99}}] \label{hovprop} Let $\C$ be a category with small limits and colimits. Suppose $\W$ is a subcategory of $\C$ and $I$ and $J$ are sets of morphisms of $\C$. Assume that the following conditions are satisfied:

{\rm (i)} The subcategory $\W$ satisfies the two out of three property and is closed under retracts.

{\rm (ii)} The domains of $I$ and $J$ are small relative to $I$-cell and $J$-cell, respectively.

{\rm (iii)} $J$-cell $\subset$ $\W \cap I$-cof.

{\rm (iv)} $I$-inj = $\W \cap J$-inj.
 
Then $\C$ is a cofibrantly generated model category with $\W$ the class of weak equivalences, $J$-inj the class fibrations and $I$-cof the class of cofibrations. 

\end{prop}

Note that the set $I$ is usually referred to as a  set of generating cofibrations and $J$ as a set of generating acyclic cofibrations. 

\;

\;

Further, we recall the definitions of monoidal model categories and enriched model categories.

\begin{defi}[{see e.g. \cite[Definition 4.2.6]{Hov99}}] A \emph{monoidal model category} is a closed symmetric monoidal category $\V$ together with a model structure such that the following conditions hold:

{\rm (i)} (The pushout-product axiom) Let $ i \colon K \lra L$ and $ j \colon A \lra B$ be cofibrations in the model category $\V$. Then the induced map
$$i \boxempty j \colon K \wedge B \bigvee_{K \wedge A} L \wedge A \lra L \wedge B$$ 
is a cofibration in $\V$. Furthermore, if either $i$ or $j$ is an acyclic cofibration, then so is $i \boxempty j$.

{\rm (ii)} Let $q \colon QI \lra I$ be a cofibrant replacement for the unit $I$. Then the maps
$$q \wedge 1 \colon QI \wedge X \lra I \wedge X \;\; \text{and} \;\; 1 \wedge q \colon X \wedge QI \lra X \wedge I$$
are weak equivalences for any cofibrant $X$. \end{defi}

\begin{defi} [{see e.g. \cite[Definition 4.2.18]{Hov99}}] \label{Vmod} Let $\V$ be a monoidal model category. A \emph{$\V$-model category} is a model category $\C$ with the following data and properties:

{\rm (i)} The category $\C$ is enriched, tensored and cotensored over $\V$ (see \cite[Section 1.2 and Section 3.7]{Kel05}). This means that we have \emph{tensors} $K \wedge X$ and \emph{cotensors} $X^K$ and mapping objects $\Hom(X,Y) \in \V$ for $K \in \V$ and $X, Y \in \C$ and all these functors are related by  $\V$-enriched adjunctions
$$\Hom(K \wedge X, Y) \cong \Hom(X, Y^K) \cong \Hom(K, \Hom(X,Y)). $$

{\rm (ii)} (The pushout-product axiom) Let $ i \colon K \lra L$ be a cofibration in the model category $\V$ and $ j \colon A \lra B$ a cofibration in the model category $\C$. Then the induced map
$$i \boxempty j \colon K \wedge B \bigvee_{K \wedge A} L \wedge A \lra L \wedge B$$ 
is a cofibration in $\C$. Furthermore, if either $i$ or $j$ is an acyclic cofibration, then so is $i \boxempty j$.

{\rm (iii)} If $q \colon QI \lra I$ is a cofibrant replacement for the unit $I$ in $\V$, then the induced map $q \wedge 1 \colon QI \wedge X \lra I \wedge X$ is a weak equivalence in $\C$ for any cofibrant $X$.  \end{defi}

Finally, let us recall the definition of a proper model category.

\begin{defi} A model category is called \emph{left proper} if weak equivalences are preserved by pushouts along cofibrations. Dually, a model category is called \emph{right proper} if weak equivalences are preserved by pullbacks along fibrations. A model category which is left proper and right proper is said to be \emph{proper}. \end{defi}

\subsection{G-equivariant spaces} \label{Gspace}

\begin{conv} In this paper $G$ will always denote a finite group. \end{conv}

\begin{conv} By a topological space we will always mean a compactly generated weak Hausdorff space.\end{conv}

The category $G \Top_*$ of pointed topological $G$-spaces admits a proper and cofibrantly generated model structure such that $f \colon X \lra Y$ is a weak equivalence (resp. fibration) if the induced map on $H$-fixed points 
$$f^H \colon X^H \lra Y^H$$ 
is a weak homotopy equivalence (resp. Serre fibration) for any subgroup $H \leq G$ (see e.g. \cite[III.1]{MM02}). The set
$$(G/H \times S^{n-1})_+ \lra (G/H \times D^n)_+, \; \;\; n \geq 0 \;\;\;, H \leq G $$
of $G$-maps generates cofibrations in this model structure. The acyclic cofibrations are generated by the maps
$$ \incl_0 \colon (G/H \times D^n)_+ \lra (G/H \times D^n \times I)_+, \; \;\; n \geq 0 \;\;\;, H \leq G. $$

The model category $G \Top_*$ is a closed symmetric monoidal model category \cite[III.1]{MM02}. The monoidal product on $G \Top_*$ is given by the smash product $X \wedge Y$,   with the diagonal $G$-action, for any $X, Y \in G \Top_*$, and the mapping object is the nonequivariant pointed mapping space $\Map(X,Y)$ with the conjugation $G$-action. 

\subsection{G-equivariant orthogonal spectra} \label{GorthoII} We start by reminding the reader about the definition of an orthogonal spectrum \cite{MMSS}:

\begin{defi} \label{orthoI} An \emph{orthogonal spectrum} $X$ consists of the following data:

$\quad \bullet$ a sequence of pointed spaces $X_n$, for $n \geq 0$;

$\quad \bullet$ a base-point preserving continuous action of the orthogonal group $O(n)$ on $X_n$ for each $n \geq 0$;

$\quad \bullet$ continuous based maps $\sigma_n \colon X_n \wedge S^1 \lra X_{n+1}$.

This data is subject to the following condition: For all $n,m \geq 0$, the iterated structure map $X_n \wedge S^m \lra X_{n+m}$ is $O(n) \times O(m)$-equivariant.
 
\end{defi}

Next, let us recall the definition of $G$-equivariant orthogonal spectra (here we mainly follow \cite{Scheq}. See also \cite{MM02} which is the original source for $G$-equivariant orthogonal spectra): 

\begin{defi} \label{GorthoI} An \emph{orthogonal} $G$-\emph{spectrum} (\emph{$G$-equivariant orthogonal spectrum}) is an orthogonal spectrum $X$ equipped with a categorical $G$-action, i.e., with a group homomorphism $G \lra \Aut(X)$. \end{defi}

The category of orthogonal $G$-spectra is denoted by $\Sp_G^O$. Any orthogonal $G$-spectrum $X$ can be evaluated on an arbitrary finite dimensional orthogonal $G$-representation $V$. The $G$-space $X(V)$ is defined by 
$$X(V) = \LL(\RR^n, V)_+ \wedge_{O(n)} X_n,$$
where the number $n$ is the dimension of $V$, the vector space $\RR^n$ is equipped with the standard scalar product and $\LL(\RR^n, V)$ is the space of (not necessarily equivariant) linear isometries from $\RR^n$ to $V$. The $G$-action on $X(V)$ is given diagonally:
$$g \cdot [\varphi, x]= [g \varphi , g x], \;\; g \in G, \;\; \varphi \in \LL(\RR^n, V), \;\; x \in X_n.$$
For the trivial $G$-representation $\RR^n$, the pointed $G$-space $X(\RR^n)$ is canonically isomorphic to the pointed $G$-space $X_n$. Next, let $S^V$ denote the representation sphere of $V$, i.e., the one-point compactification of $V$. Using the iterated structure maps of $X$, for any finite dimensional orthogonal $G$-representations $V$ and $W$, one can define $G$-equivariant \emph{generalized structure maps} 
$$\sigma_{V,W} \colon X(V) \wedge S^W \lra X(V \oplus W).$$ 
These are then used to define $G$-\emph{equivariant homotopy groups} 
$$\pi_k^G X = \colim_n{[S^{k+n\rho_G}, X(n\rho_G)]}^G, \;\; k \in \Z,$$
where $\rho_G$ denotes the regular representation of $G$. Furthermore, for any subgroup $H \leq G$, one defines $\pi_k^HX,\; k \in \Z$, to be the $k$-th $H$-equivariant homotopy group of $X$ considered as an $H$-spectrum.

\begin{defi} \label{GorthoIII} A map $f \colon X \lra Y$ of $G$-equivariant orthogonal spectra is called a \emph{stable equivalence} if the induced map
$$\pi_k^H(f) \colon \pi_k^H X \lra \pi_k^H Y$$
is an isomorphism for any integer $k$ and any subgroup $H \leq G$.
 
\end{defi}

\subsection{Comparison of different definitions} \label{difdef}

Before continuing the recollection, let us explain the relation of Definition \ref{GorthoI} with the orginal definition of $G$-equivariant  orthogonal spectra due to Mandell and May. For this we first recall the category $O_G$. The objects of $O_G$ are finite dimensional orthogonal $G$-representations. For any orthogonal $G$-representations $V$ and $W$, the pointed morphism $G$-space $O_G(V,W)$ is defined to be the Thom space of the $G$-equivariant vector bundle
$$\xi(V, W) \lra \LL(V,W),$$
where $\LL(V,W)$ is the space of linear isometric embeddings from $V$ to $W$ and 
$$\xi(V, W) = \{(f, x) \in \LL(V,W) \times W | x \perp f(V) \}.$$
For more details about this category see \cite[II.4]{MM02}. It follows from \cite[Theorem V.1.5]{MM02} (see also \cite[Proposition A.18]{HHR}) that the category of $O_G$-spaces (which is the category of $G \Top_*$-enriched functors from $O_G$ to $G \Top_*$) is equivalent to the category of $G$-equivariant orthogonal spectra.

\begin{remk} In \cite[II.2]{MM02}, Mandell and May define $G$-equivariant orthogonal spectra indexed on a \emph{$G$-universe} $\U$ (which is a countably infinite dimensional real inner product space with certain properties \cite[Definition II.1.1]{MM02}). Such a $G$-spectrum is a collection of $G$-spaces indexed on those representations that embed into $\U$ together with certain equivariant structure maps. It follows from \cite[Theorem II.4.3, Theorem V.1.5]{MM02} that for any $G$-universe $\U$, the category of orthogonal $G$-spectra indexed on $\U$ and the category $\Sp^O_G$ are equivalent. This shows that universes are not really relevant for the point-set level definition of an orthogonal $G$-spectrum. However, they become really important when one considers the homotopy theory of orthogonal $G$-spectra. We will use the homotopy theory of orthogonal $G$-spectra where all finite dimensional orthogonal G-representations are built in. This means that we will work with the \emph{genuine $G$-spectra} or in other words with the homotopy theory of orthogonal $G$-spectra indexed on a complete universe (as a model of such a universe one can take the sum $\infty \rho_G$ of countable copies of the regular representation $\rho_G$). \end{remk}
 
Next, the category $\Sp^O_G$ is a closed symmetric monoidal category. The symmetric monoidal structure on $\Sp^O_G$ is given by the smash product of underlying orthogonal spectra \cite{MMSS} with the diagonal $G$-action. Further, for any universe $\U$, the category of $G$-equivariant orthogonal spectra indexed on $\U$ as well as the category of $O_G$-spaces are closed symmetric monoidal categories. It follows from \cite[Theorem II.4.3, Theorem V.1.5]{MM02} (see also \cite[Proposition A.18]{HHR}) that all the equivalences discussed above are in fact equivalences of closed symmetric monoidal categories.  

From this point on we will freely use all the results of \cite{MM02} for the category $\Sp^O_G$ having the above equivalences in mind.

\subsection{The level model structure on $\Sp^O_G$} \label{Gortholev} In this subsection we closely follow \cite[III.2]{MM02}. 

For any finite dimensional orthogonal $G$-representation $V$, the evaluation functor 
$$\Ev_V \colon \Sp^O_G \lra  G \Top_*,$$ 
given by $X \mapsto X(V)$, has a left adjoint $G \Top_*$-functor
$$F_V \colon G \Top_* \lra \Sp^O_G$$
which is defined by (see \cite[II.4]{MM02})
$$F_VA(W) = O_G(V, W) \wedge A.$$
We fix (once and for all) a small skeleton $\sk O_G$ of the category $O_G$. Let $I^{G}_{\lv}$ denote the set of morphisms
$$\{F_V (G/H \times S^{n-1})_+ ) \lra F_V((G/H \times D^n)_+) \; | \; V \in \sk O_G, \;\; n \geq 0, \;\; H \leq G \}$$
and $J^{G}_{\lv}$ denote the set of morphisms
$$\{F_V ((G/H \times D^n)_+) \lra F_V((G/H \times D^n \times I)_+) \; | \; V \in \sk O_G, \;\; n \geq 0, \;\; H \leq G \}.$$
In other words, the sets $I^{G}_{\lv}$ and $J^{G} _{\lv}$  are obtained by applying the functors $F_V$, $V \in \sk O_G$, to the generating cofibrations and generating acyclic cofibrations of $G \Top_*$, respectively. Further, we recall

\begin{defi} Let $ f \colon X \lra Y$ be a morphism in $\Sp^O_G$. The map $f$ is called a \emph{level equivalence} if $f(V) \colon X(V) \lra Y(V)$ is a weak equivalence in $G \Top_*$ for any $V \in \sk O_G$. It is called a \emph{level fibration} if $f(V) \colon X(V) \lra Y(V)$ is a fibration in $G \Top_*$ for any $V \in \sk O_G$. A map in $\Sp^O_G$ is called a \emph{cofibration} if it has the left lifting property with respect to all maps that are level fibrations and level equivalences (i.e., level acyclic fibrations). 
 
\end{defi}

\begin{prop}[{\cite[III.2.4]{MM02}}] The category $\Sp^O_G$ together with level equivalences, level fibrations and cofibrations forms a cofibrantly generated, proper model category. The set $I^{G}_{\lv}$ serves as a set of generating cofibrations and the set $J^{G}_{\lv}$ serves as a set of generating acyclic cofibrations.
 
\end{prop}

\subsection{The stable model structure on $\Sp^O_G$} \label{Gorthostab} The reference for this subsection is \cite[III.4]{MM02}. 

Recall that for any $G$-equivariant orthogonal spectrum $X$ we have the generalized structure maps
$$\sigma_{V, W} \colon  X(V) \wedge S^W \lra X(V \oplus W).$$
Let $\wt{\sigma}_{V, W} \colon X(V) \lra \Omega^{W}X(V \oplus W)$ denote the adjoint of $\sigma_{V, W}$.

\begin{defi} \label{Gomega} An orthogonal $G$-spectrum $X$ is called a $G$-$\Omega$-spectrum if the maps $\wt{\sigma}_{V, W}$ are weak equivalences in $G \Top_*$ for any $V$ and $W$ in $O_G$. \end{defi}

Before formulating the theorem about the stable model structure on $\Sp^O_G$, let us recall certain morphisms in $\Sp^O_G$ that will form a generating set of acyclic cofibrations for this model structure. Let $V, W \in \sk O_G$ and
$$\lambda_{V,W} \colon F_{V \oplus W}S^W \lra F_VS^0$$ 
denote the map of $G$-equivariant orthogonal spectra that is adjoint to the map
$$S^W \lra \Ev_{V \oplus W}(F_VS^0) = O_G(V, V \oplus W)$$
which sends $z \in W$ to $(\xymatrix{V \; \ar@{^{(}->}[r]^-{(1, 0)} & V \oplus W}, z)$ (see \cite[III.4.3]{MM02}). Using the mapping cylinder construction, the map $\lambda_{V,W}$ factors as a composite
$$\xymatrix{F_{V \oplus W}S^W \ar[r]^{\kappa_{V,W}} & M\lambda_{V \oplus W} \ar[r]^{r_{V \oplus W}} & F_VS^0,}$$
where $r_{V \oplus W}$ is a $G$-equivariant homotopy equivalence and $\kappa_{V,W}$ a cofibration and a stable equivalence \cite[III.4.5-4.6]{MM02}. Now consider any generating cofibration 
$$i \colon (G/H \times S^{n-1})_+  \lra (G/H \times D^n)_+.$$ 
Let $i \boxempty \kappa_{V, W}$ denote the pushout-product induced from the commutative square:
$$\xymatrix{(G/H \times S^{n-1})_+  \wedge F_{V \oplus W}S^W  \ar[r] \ar[d] & (G/H \times S^{n-1})_+ \wedge  M\lambda_{V \oplus W} \ar[d] & \\  (G/H \times D^n)_+ \wedge F_{V \oplus W}S^W \ar[r] & (G/H \times D^n)_+ \wedge  M\lambda_{V \oplus W}.}$$ 
Define
$$K^{G} = \{i \boxempty \kappa_{V, W} \; | \; H \leq G, \;\; n \geq 0, \;\; V, W \in \sk O_G \}.$$
Let $J_{\st}^{G}$ stand for the union $J_{\lv}^{G} \cup K^{G}$. For convenience, we will also introduce the notation $I^{G}_{\st}= I^{G}_{\lv}$. 

Finally, before formulating the main theorem of this subsection we need the following definition:

\begin{defi} A map $f \colon X \lra Y$ of orthogonal $G$-spectra is called a \emph{stable fibration}, if it has the right lifting property with respect to the maps that are cofibrations and stable equivalences. \end{defi}

\begin{theo} [{\cite[III.4.2]{MM02}}] The category $\Sp^O_G$ together with cofibrations, stable equivalences and stable fibrations forms a proper, cofibrantly generated, stable model category. The set $I^{G}_{\st}$ generates cofibrations and the set $J_{\st}^{G}$ generates acyclic cofibrations. Furthermore, the fibrant objects are precisely the $G$-$\Omega$-spectra. \end{theo}

The category $\Sp^O_G$ together with the latter model structure is referred to as the \emph{stable model category of orthogonal $G$-spectra (indexed on a complete $G$-universe}). From now on the symbol $\Sp^O_G$ will always stand for this model category.

Finally, we recall that the stable model category $\Sp^O_G$ together with the smash product forms a closed symmetric monoidal model category \cite[III.7]{MM02}. In particular, the following holds:

\begin{prop} \label{Gorthomon} Suppose that $i \colon K \lra L$ and $j \colon A \lra B$ are cofibrations in $\Sp^O_G$. Then the pushout-product
$$i \boxempty j \colon K \wedge B \bigvee_{K \wedge A} L \wedge A \lra L \wedge B$$
is a cofibration in $\Sp^O_G$. The map $i \boxempty j$ is also a stable equivalence if in addition $i$ or $j$  is a stable equivalence.
 
\end{prop}

\subsection{The equivariant stable homotopy category} \label{hosp} In this subsection we list some well known properties of the homotopy category $\Ho(\Sp^O_G)$. Note that the category $\Ho(\Sp^O_G)$ is equivalent to the Lewis-May $G$-equivariant stable homotopy category of genuine $G$-spectra (see \cite[IV.1]{MM02}) introduced in \cite{LMS}.

As noted in the previous subsection, the model category $\Sp^O_G$ is stable and hence the homotopy category $\Ho(\Sp^O_G)$ is naturally triangulated. Further, since the maps 
$$\lambda_V=\lambda_{0, V} \colon F_VS^V \lra F_0S^0$$ 
are stable equivalences \cite[Lemma III.4.5]{MM02}, it follows that the functor
$$S^V \wedge - \colon \Ho(\Sp^O_G) \lra \Ho(\Sp^O_G)$$
is an equivalence of categories for any finite dimensional orthogonal $G$-representation $V$.

\;

\;

Next, before continuing, let us introduce the following notational convention. For any $G$-equivariant orthogonal spectra $X$ and $Y$, the abelian group $[X,Y]^{\Ho(\Sp^O_G)}$ of morphisms from $X$ to $Y$ in $\Ho(\Sp^O_G)$ will be denoted by $[X,Y]^G$. 

\;

\;

An adjunction argument immediately implies that for any subgroup $H \leq G$ and an orthogonal $G$-spectrum $X$, there is a natural isomorphism
$$[\Sigma_{+}^\infty G/H, X]^G_* \cong \pi_*^H X.$$
As a consequence, we see that the set
$$\{\Sigma_{+}^\infty G/H \; | \; H \leq G \}$$
is a set of compact generators for the triangulated category $\Ho(\Sp^O_G)$. Note that since $G$ is finite, for $\ast > 0$ and any subgroups $H, H' \leq G$, the abelian group
$$[\Sigma_{+}^\infty G/H, \Sigma_{+}^\infty G/H']_*^G$$
is finite (see e.g. \cite[Proposition A.3]{GM95}). 

Finally, recall from the introduction that the stable Burnside category which is the full preadditive subcategory of $\Ho(\Sp_G^O)$ with objects the stable orbits $\Sigma_{+}^{\infty}G/H$, $H \leq G$, is generated by the conjugations, transfers and restrictions. The stable Burnside category plays an important role in equivariant stable homotopy theory as well as in representation theory. The contravariant functors from this category to abelian groups are exactly \emph{Mackey functors}. Note that the stable Burnside category shows up in the formulation and proof of Theorem \ref{damtkicebuli}.   

\section{Categorical Input} \label{catinp}

\setcounter{subsection}{0}

\subsection{Outline} \label{outl}

Recall that $G$ is a finite group. We start with 

\begin{defi} \label{eqmodel} A a $G \Top_*$-model category $\C$ (see Definition \ref{Vmod} and Subsection \ref{Gspace}) is said to be a $G$-\emph{equivariant stable model category} if the adjunction
$$\xymatrix{ S^V \wedge -  \colon \C \ar@<0.5ex>[r] & \C \cocolon \Omega^V(-)=(-)^{S^V} \ar@<0.5ex>[l]}$$
is a Quillen equivalence for any finite dimensional orthogonal $G$-representation $V$.

\end{defi}

Examples of $G$-equivariant stable model categories are the model category $\Sp_G^O$ of $G$-equivariant orthogonal spectra
\cite[II-III]{MM02}, the model category of $G$-equivariant orthogonal spectra equipped with the $\S$-model structure \cite{Sto11}, the model category of $S_G$-modules \cite[IV.2]{MM02}, the model category of $G$-equivariant continuous functors \cite{Blum05} and the model categories of $G$-equivariant topological symmetric spectra (\cite{Man04}, \cite{Haus}).

\;

\;

The following proposition is an equivariant version of \cite[Theorem 3.8.2]{SS03}.

\begin{prop}\label{spectrific} Let $\C$ be a cofibrantly generated (Definition \ref{Gcofgen}), proper, $G$-equivariant stable model category. Then the category ${ \Sp}^O_G(\C)$ of internal orthogonal $G$-spectra in $\C$ (Definition \ref{Gmodcat}) possesses a $G$-equivariant stable model structure and the $G \Top_*$-adjunction
$$\xymatrix{ \Sigma^{\infty}  \colon \C \ar@<0.5ex>[r] & { \Sp}^O_G(\C) \cocolon \Ev_0 \ar@<0.5ex>[l]}$$
is a Quillen equivalence.

\end{prop}

The proof of this proposition is a straightforward equivariant generalization of the arguments in \cite[3.8]{SS03}. It will occupy a significant part of this section.

The point of Proposition \ref{spectrific} is that one can replace (under some technical assumptions) any $G$-equivariant stable model category by a $G$-spectral one (Definition \ref{GSpec}), i.e., by an $\Sp_G^O$-model category. This in particular implies that $\Ho(\C)$ is tensored over the $G$-equivariant stable homotopy category $\Ho( \Sp_G^O)$. 

To stress the importance of Proposition \ref{spectrific}, we will now give a general strategy for how one should try to prove Conjecture \ref{mainconj}. Recall that we are given a triangulated equivalence 
$$\xymatrix{\Psi \colon \Ho( \Sp_G^O) \ar[r]^-{\sim} & \Ho(\C)}$$
with certain properties. By Proposition \ref{spectrific}, there is a $G \Top_*$-Quillen equivalence 
$$\xymatrix{ \Sigma^{\infty}  \colon \C \ar@<0.5ex>[r] & {\Sp}^O_G(\C) \cocolon \Ev_0. \ar@<0.5ex>[l]}$$
Let $X$ be a cofibrant replacement of $(\mathbf{L}\Sigma^{\infty} \circ \Psi) (\S)$. Since $\Sp^O_G(\C)$ is $G$-spectral (Definition \ref{GSpec}), there is a $G \Top_*$-Quillen adjunction 
$$\xymatrix{-\wedge X \colon \Sp_G^O \ar@<0.5ex>[r] & {\Sp}^O_G(\C) \cocolon \Hom(X,-).  \ar@<0.5ex>[l]}$$
Hence, in order to prove Conjecture \ref{mainconj}, it suffices to show that the latter Quillen adjunction is a Quillen equivalence. Next, it follows from the properties of $\Psi$ that we have isomorphisms
$$\Psi(\Sigma^{\infty}_+G/H) \cong \R \Ev_0 (\Sigma^{\infty}_+G/H \wedge^{\LL} X)$$
which are natural with respect to transfers, conjugations, and restrictions. Using these isomorphisms, we can choose an inverse of $\Psi$
$$\Psi^{-1} \colon \Ho(\C) \lra \Ho(\Sp^O_G)$$
such that $\Psi^{-1}(\R \Ev_0 (\Sigma^{\infty}_+G/H \wedge^{\LL} X))=\Sigma^{\infty}_+G/H$. Moreover, since the isomorphisms above are natural with respect to the maps in the stable Burnside category, we get the identities
$$\Psi^{-1} (\R \Ev_0 (g\wedge^{\LL} X)) = g, \; \;  \Psi^{-1}( \R \Ev_0 (\res_{K}^H \wedge^{\LL} X))=\res_{K}^H, \; \; \Psi^{-1}(\R \Ev_0 (\tr_{K}^H \wedge^{\LL} X))=\tr_{K}^H,$$
where $g \in G$ and $K \leq H \leq G$. Now let us consider the composite 
$$\xymatrix{ F \colon \Ho(\Sp_G^O) \ar[r]^-{-\wedge^{\mathbf{L}} X} & \Ho(\Sp^O_G(\C)) \ar[r]^-{\mathbf{R}\Ev_0} & \Ho(\C) \ar[r]^-{\Psi^{-1}} & \Ho(\Sp_G^O).}$$
Since the functors $\mathbf{R}\Ev_0$ and $\Psi^{-1}$ are equivalences, to prove that $(-\wedge X, \Hom(X,-))$ is a Quillen equivalence is equivalent to showing that the endofunctor 
$$\xymatrix{ F \colon \Ho(\Sp_G^O) \ar[r] & \Ho(\Sp_G^O)}$$ 
is an equivalence of categories. By the assumptions of Conjecture \ref{mainconj} and the properties of $\Psi^{-1}$, we see that $F$ enjoys the following properties:

\;

\;

{\rm (i)} $F(\Sigma_{+}^{\infty} G/H) = \Sigma_{+}^{\infty} G/H$, $H \leq G$; 

\;

\;

{\rm (ii)} $F$ preserves transfers, conjugations, and restrictions (and hence the stable Burnside category);

\;

\;

{\rm (iii)} $F$ is an exact functor of triangulated categories and preserves infinite coproducts.

\;

\;

Similarly, if we start with the $2$-localized genuine $G$-equivariant stable homotopy category $\Ho(\Sp^O_G,_{(2)})$ and an equivalence $\Ho(\Sp^O_G,_{(2)}) \sim \Ho(\C)$ as in the formulation of Theorem \ref{damtkicebuli}, we obtain an endofunctor $\xymatrix{\Ho(\Sp_G^O,_{(2)}) \ar[r] & \Ho(\Sp_G^O,_{(2)})}$ which also satisfies the properties (i), (ii) and (iii) above. The following theorem which is one of the central results of this paper, immediately implies Theorem \ref{damtkicebuli}:

\begin{theo} \label{reduction} Let $G$ be a finite group and $\xymatrix{ F \colon \Ho(\Sp_G^O,_{(2)}) \ar[r] & \Ho(\Sp_G^O,_{(2)})}$ an exact functor of triangulated categories that preserves arbitrary coproducts and such that 
$$F(\Sigma_{+}^{\infty} G/H) = \Sigma_{+}^{\infty} G/H, \;\; H \leq G,$$ 
and
$$F(g)=g, \;\; F(\res_{K}^H)= \res_{K}^H, \;\; F(\tr_{K}^H)=\tr_{K}^H, \;\; g \in G, \;\; K \leq H \leq G.$$
Then $F$ is an equivalence of categories. \end{theo}

The proof of this theorem will be completed at the very end of this paper. In this section we will concentrate on the proof of Proposition \ref{spectrific} and on the $p$-localization of the stable model category $\Sp_G^O$.

\;

\;

Before starting the preparation for the proof of Proposition \ref{spectrific}, let us outline the plan that will lead to the proof of Proposition \ref{spectrific}. We first define the category $\Sp^O_G(\C)$ of orthogonal $G$-spectra internal to  a $G \Top_*$-model category $\C$ and discuss its categorical properties. Next, for any cofibrantly generated $G \Top_*$-model category $\C$ we construct the level model structure on $\Sp^O_G(\C)$. Finally, using the same strategy as in \cite{SS03}, we establish the $G$-equivariant stable model structure on $\Sp^O_G(\C)$ for any proper, cofibrantly generated, $G \Top_*$-model category $\C$ that is stable as an ordinary model category.

\subsection{Orthogonal $G$-spectra in equivariant model categories} \label{orthmodI}

Recall from Subsection \ref{difdef} the $G \Top_*$-category $O_G$. The objects of $O_G$ are finite dimensional orthogonal $G$-representations. For any finite dimensional orthogonal $G$-representations $V$ and $W$, the pointed morphism $G$-space from $V$ to $W$ is the Thom space $O_G(V,W)$.
Recall also that the category $\Sp_G^O$ is equivalent to the category of $O_G$-spaces (which is the category of $G \Top_*$-enriched functors from $O_G$ to $G \Top_*$). 
  
Now suppose that $\C$ is a $G \Top_*$-model category (in particular, $\C$ is pointed). We remind the reader that this means that we have tensors $K \wedge X$, cotensors $X^K$ and pointed mapping $G$-spaces $\Map(X,Y)$ for $K \in G \Top_*$ and $X, Y \in \C$, which are related by adjunctions and satisfy certain properties (Definition \ref{Vmod}). In particular, the pushout-product axiom holds: Let $ i \colon K \lra L$ be a cofibration in the model category $G \Top_*$ and $ j \colon A \lra B$ a cofibration in the model category $\C$. Then the induced map
$$i \boxempty j \colon K \wedge B \bigvee_{K \wedge A} L \wedge A \lra L \wedge B$$ 
is a cofibration in $\C$. Furthermore, if either $i$ or $j$ is an acyclic cofibration, then so is $i \boxempty j$.

\begin{defi} \label{Gmodcat} Let $\C$ be a $G \Top_*$-model category. An \emph{orthogonal $G$-spectrum} in $\C$ is a $G \Top_*$-enriched functor (\cite[1.2]{Kel05}) from the category $O_G$ to $\C$. \end{defi}
 
The category of orthogonal $G$-spectra in $\C$ will be denoted by $\Sp^O_G(\C)$. Note that by \cite[II.4.3]{MM02} (see also Subsection \ref{difdef}), the category $\Sp^O_G(G \Top_*)$ is equivalent to $\Sp^O_G$. Next, since $\C$ is complete and cocomplete, so is the category $\Sp^O_G(\C)$ (see \cite[3.3]{Kel05}) and limits and colimits are constructed levelwise.

\begin{remk} \label{setth} The category $O _G$ is skeletally small. We can fix once and for all a small skeleton of $O _G$. In particular, when talking about ends and coends over $O_G$ and using the notations $\int_{V \in O_G}$ and $\int^{V \in O_G}$, we will always implicitly mean that the indexing category is the chosen small skeleton of $O _G$. \end{remk}

\begin{prop} \label{sp_gmodcat} Let $\C$ be a $G \Top_*$-model category. The category $\Sp^O_G(\C)$ is enriched, tensored and cotensored over the symmetric monoidal category $\Sp^O_G$ of equivariant orthogonal $G$-spectra. \end{prop}

\begin{proof} Let $K \in \Sp^O_G$ and $X \in \Sp^O_G(\C)$. We define $ K \wedge X \in \Sp^O_G(\C)$ by the following $G \Top_*$-enriched coend
$$K \wedge X = \int^{V, W \in O_G} O_G(V \oplus W, -) \wedge K(V) \wedge X(W).$$
This product is unital and coherently associative. The proof uses the enriched Yoneda Lemma \cite[Section 3.10, (3.71)]{Kel05} and the Fubini theorem \cite[Section 3.10, (3.63)]{Kel05}. We do not provide the details here as they are standard and well-known. Next, one defines cotensors by a $G \Top_*$-enriched end
$$X^K (V) = \int_{W \in O_G}  X(W \oplus V)^{K(W)}.$$
Finally, for any $X, Y \in \Sp^O_G(\C)$, one can define Hom-$G$-spectra by a $G \Top_*$-enriched end
$$\Hom(X, Y)(V) = \int_{W \in O_G} \Map(X(W), Y(W\oplus V)).$$
It is an immediate consequence of \cite[Section 3.10, (3.71)]{Kel05} that these functors satisfy all the necessary adjointness properties:
\[\Hom(K \wedge X, Y) \cong \Hom(X, Y^K) \cong \Hom(K, \Hom(X,Y)). \qedhere \] \end{proof}

\subsection{The level model structure on $\Sp^O_G(\C)$}  \label{levelsect}

In order to establish the stable model structure on $\Sp^O_G(\C)$, one needs the additional assumption that $\C$ is a cofibrantly generated $G \Top_*$-model category. 

\begin{defi} \label{Gcofgen} Let $\C$ be a $G \Top_*$-model category. We say that $\C$ is a \emph{cofibrantly generated $G \Top_*$-model category}, if there are sets $I$ and $J$ of maps in $\C$ such that the following hold:
 
{\rm (i)} Let $A$ be the domain or codomain of a morphism from $I$. Then for any subgroup $H \leq G$ and any $n \geq 0$, the object
$$(G/H \times D^n)_+ \wedge A$$
is small relative to $I$-cell (and hence relative to $I$-cof by \cite[2.1.16]{Hov99}).

{\rm (ii)} Domains of morphisms in $J$ are small relative to $J$-cell and $I$-cell.

{\rm (iii)} The class of fibrations is $J$-inj.

{\rm (iv)} The class of acyclic fibrations is $I$-inj.

\end{defi}

The model category $G \Top_*$ is a cofibrantly generated $G \Top_*$-model category \cite[Theorem III.1.8]{MM02}. Other important examples of cofibrantly generated $G \Top_*$-model categories are the model category $\Sp_G^O$ of $G$-equivariant orthogonal spectra \cite[Theorem III.4.2]{MM02}, the model category of $G$-equivariant orthogonal spectra equipped with the $\S$-model structure \cite[Theorem 2.3.27]{Sto11}, the model category of $S_G$-modules \cite[Theorem IV.2.8]{MM02}, the model category of $G$-equivariant continuous functors \cite[Theorem 1.3]{Blum05} and the model categories of $G$-equivariant topological symmetric spectra (\cite{Man04}, \cite{Haus}).

\begin{remk} If a $G \Top_*$-model category $\C$ is cofibrantly generated as an ordinary model category (Definition \ref{cofgen}), then it doesn't necessarily follow that $\C$ is a cofibrantly generated $G \Top_*$-model category in the sense of Definition \ref{Gcofgen}. 

The conditions in  Definition \ref{Gcofgen} are essentially needed for the proof of Proposition \ref{stmod}. In fact, all the claims in this section that come before Proposition \ref{stmod} do not really use that $\C$ satisfies all the conditions of Definition \ref{Gcofgen}. They still hold if we only assume that  $\C$ is a $G \Top_*$-model category and cofibrantly generated as an underlying model category. However, for the rest of the paper, we decided to concentrate only on cofibrantly generated $G \Top_*$-model categories in the sense of Definition \ref{Gcofgen} since more general model categories are irrelevant here. \end{remk}

Now suppose that $\C$ is a cofibrantly generated $G \Top_*$-model category with $I$ and $J$ generating cofibrations and acyclic cofibrations

\begin{defi} \label{levstr} Let $ f \colon X \lra Y$ be a morphism in $\Sp^O_G(\C)$. The map $f$ is called a \emph{level equivalence} if $f(V) \colon X(V) \lra Y(V)$ is a weak equivalence in $\C$ for any $V \in O_G$. It is called a \emph{level fibration} if $f(V) \colon X(V) \lra Y(V)$ is a fibration in $\C$ for any $V \in O_G$. A map in $\Sp^O_G(\C)$ is called a \emph{cofibration} if it has the left lifting property with respect to all maps that are level fibrations and level equivalences (i.e., level acyclic fibrations). \end{defi}

The level model structure on $\Sp^O_G(\C)$ which we will construct now is a cofibrantly generated model structure. Before stating the main proposition of this subsection we would like to introduce the set of morphisms that will serve as generators of (acyclic) cofibrations in the level model structure on $\Sp^O_G(\C)$.

The evaluation functor $\Ev_V \colon \Sp^O_G(\C) \lra \C$, given by $X \mapsto X(V)$, has a left adjoint $G \Top_*$-functor
$$F_V \colon \C \lra \Sp^O_G(\C)$$
which is defined by 
$$F_VA = O_G(V, -) \wedge A.$$
For any finite dimensional orthogonal $G$-representation $V$, consider the following sets of morphisms
$$F_VI = \{ F_Vi \; | \; i \in I\} \;\;\;\; \text{and} \;\;\;\; F_VJ = \{ F_Vj \; | \; j \in J\}.$$
Next, fix (once and for all) a small skeleton $\sk O_G$ of the category $O_G$. We finally define
$$FI = \bigcup_{V \in \sk O_G} F_VI \;\;\;\; \text{and} \;\;\; FJ= \bigcup_{V \in \sk O_G} F_VJ.$$

The following proposition is an equivariant analog of \cite[Proposition 3.7.2]{SS03} (cf. \cite[Theorem 2.12]{MG11}).

\begin{prop} \label{levmod} Suppose $\C$ is a cofibrantly generated $G \Top_*$-model category. Then the category $\Sp^O_G(\C)$ of orthogonal $G$-spectra in $\C$ together with the level equivalences, cofibrations and level fibrations described in Definition \ref{levstr} forms a cofibrantly generated model category. The set $FI$ generates cofibrations and the set $FJ$ generates acyclic cofibrations. \end{prop}

\begin{proof} Illman's results \cite[Theorem 7.1, Corollary 7.2]{Ill} imply that for any finite dimensional orthogonal $G$-representations $V$ and $W$ that the space $O_G(V, W)$ is a $G$-CW complex. Since for any object $A$ in $\C$,
$$F_VA(W)= O_G(V, W) \wedge A$$
and the evaluation functors preserve colimits, it follows that any morphism in $FI$-cell is a levelwise cofibration and any morphism in $FJ$-cell is a levelwise acyclic cofibration. The rest of the proof is a verbatim translation of the proof of \cite[Proposition 3.7.2]{SS03} to our case and we don't provide the details. \end{proof}

\subsection{The stable model structure on $\Sp^O_G(\C)$}

This subsection establishes the stable model structure on $\Sp^O_G(\C)$. For this one needs more assumptions than in Proposition \ref{levmod}. More precisely, we have to assume that the cofibrantly generated $G \Top_*$-model category $\C$ is proper and stable as an ordinary model category. The strategy is to follow the arguments given in \cite[3.8]{SS03}.

Let $W$ be a finite dimensional orthogonal $G$-representation and
$$\lambda_W=\lambda_{0, W} \colon F_WS^W \lra F_0S^0=\S$$ 
denote the stable equivalence of $G$-equivariant orthogonal spectra that is adjoint to the identity map
$$\id \colon S^W \lra \Ev_W(\S) = S^W$$
(see \cite[III.4.3, III.4.5]{MM02} or Subsection \ref{Gorthostab}). 

\begin{defi} \label{omega} Let $\C$ be a $G \Top_*$-model category. An object $Z$ of $\Sp^O_G(\C)$ is called an $\Omega$-spectrum if it is level fibrant and for any finite dimensional orthogonal $G$-representation $W$, the induced map
$$\lambda_W^* \colon Z \cong Z^{F_0S^0} \lra Z^{F_WS^W}$$
is a level equivalence.
 
\end{defi}

Since $Z^{F_WS^W} \cong Z(W \oplus -)^{S^W},$ this definition recovers the definition of a $G$-$\Omega$-spectrum in the sense of \cite[Definition III.3.1]{MM02} when $\C = G \Top_*$ (see also Definition \ref{Gomega}).

Now suppose again that $\C$ is a cofibrantly generated $G \Top_*$-model category. By Proposition \ref{levmod}, the level model structure on $\Sp^O_G(\C)$ is cofibrantly generated. Hence we can choose (and fix once and for all) a cofibrant replacement functor 
$$(-)^c \colon \Sp^O_G(\C) \lra \Sp^O_G(\C).$$
\begin{defi} \label{stequi} A morphism $f \colon A \lra B$ in $\Sp^O_G(\C)$ is a stable equivalence if for any $\Omega$-spectrum $Z$, the map
$$\Hom(f^c, Z) \colon \Hom(B^c, Z) \lra \Hom(A^c, Z)$$
is a level equivalence of $G$-equivariant orthogonal spectra. \end{defi}

The following proposition is an equivariant analog of \cite[Proposition 3.8.5]{SS03}. Again we don't provide details here as the proof is completely analogous to the nonequivariant counterpart from \cite{SS03}.  

\begin{prop} \label{stablespec} Let $\C$ be a left proper and cofibrantly generated $G \Top_*$-model category. Suppose that $i \colon K \lra L$ is a cofibration in $\Sp_G^O$ and $j \colon A \lra B$ a cofibration in $\Sp^O_G(\C)$. Then the pushout-product
$$i \boxempty j \colon K \wedge B \bigvee_{K \wedge A} L \wedge A \lra L \wedge B$$
is a cofibration in $\Sp^O_G(\C)$. The map $i \boxempty j$ is also a stable equivalence if in addition $i$ or $j$  is a stable equivalence. \end{prop}

Next, we introduce the set $J_{\st}$ which will serve as a set of generating acyclic cofibrations for the stable model structure on $\Sp^O_G(\C)$ that we are going to establish. Let $W$ be a finite dimensional orthogonal $G$-representation. Consider the levelwise mapping cylinder $M\lambda_W$ of the map $\lambda_W \colon F_WS^W \lra F_0S^0$. The map $\lambda_W$ factors as a composite
$$\xymatrix{F_WS^W \ar[r]^{\kappa_W} & M\lambda_W \ar[r]^{r_W} & F_0S^0,}$$
where $r_W$ is a $G$-equivariant homotopy equivalence and $\kappa_W$ a cofibration and a stable equivalence \cite[III.4.5-4.6]{MM02} (see also Subsection \ref{Gorthostab}). Define
$$K = \{\kappa_W \boxempty F_Vi \; | \; i \in I, \;V , W \in \sk O_G\},$$
where $\boxempty$ is the pushout-product, $I$ is the fixed set of generating cofibrations in $\C$ (see Definition \ref{Gcofgen}) and $\sk O_G$ the fixed small skeleton of $O_G$ as in Subsection \ref{levelsect}. Next, recall from Proposition \ref{levmod} that we have sets $FI$ and $FJ$, generating cofibrations and acyclic cofibrations, respectively, in the level model structure. Define  
$$J_{\st}=FJ \cup K.$$ 
For the convenience we will denote the set $FI$ by $I_{\st}$. The cofibrations in the stable model structure on $\Sp^O_G(\C)$ will be the same as in the level model structure and thus $I_{\st}=FI$ will serve as a set of generating cofibrations for the stable model structure.

The following three propositions are again equivariant analogs of \cite[3.8.6-8]{SS03}. Once again we omit the proofs as they are very similar to those in \cite{SS03}.  

\begin{prop} \label{j-cof->st} Let $\C$ be a left proper and cofibrantly generated $G \Top_*$-model category. Then any morphism in $J_{\st}$-cell is an $I_{\st}$-cofibration (i.e., a cofibration) and a stable equivalence. \end{prop} 

\begin{prop} \label{omega2} Let $\C$ be a cofibrantly generated $G \Top_*$-model category and $X$ an object of $\Sp^O_G(\C)$. Then the map $X \lra \ast$ is $J_{\st}$-injective if and only if $X$ is an $\Omega$-spectrum. \end{prop}

\begin{prop} \label{j-inj->lev} Let $\C$ be a right proper and cofibrantly generated $G \Top_*$-model category which is stable as an ordinary model category. Then a map in $\Sp^O_G(\C)$ is $J_{\st}$-injective and a stable equivalence if and only if it is a level acyclic fibration.  \end{prop}

Finally, we are ready to establish the stable model structure. The following proposition constructs the desired model structure. The proof of the fact that this model structure is stable is postponed to the next subsection. 

\begin{prop} \label{stmod} Let $\C$ be a proper and cofibrantly generated $G \Top_*$-model category which is stable as an ordinary model category. Then the category $\Sp^O_G(\C)$ admits a cofibrantly generated model structure with stable equivalences as weak equivalences. The sets $I_{\st}$ and $J_{\st}$ generate cofibrations and acyclic cofibrations, respectively. Furthermore, the fibrant objects are precisely the $\Omega$-spectra  \end{prop}

\begin{proof} The strategy of the proof is to verify the conditions of Proposition \ref{hovprop}. The only things that still have to be checked are the smallness conditions. The rest follows from the previous three propositions. 

That the domains of morphisms from $I_{\st}$ are small relative to $I_{\st}$-cell follows from the equality $I_{\st}=FI$ and Proposition \ref{levmod}. Next, recall that $J_{\st}= FJ \cup K$. We will now verify that the domains of morphisms from $J_{\st}$ are small relative to levelwise cofibrations. This will immediately imply that the domains of morphisms in $J_{\st}$ are small relative to $J_{\st}$-cell since 
$$J_{\st}\text{-cell} \subset I_{\st} \text{-cof}$$
by Proposition \ref{j-cof->st} and any morphism in $I_{\st}$-cof is a levelwise cofibration as we saw in the proof of Proposition \ref{levmod}. That the domains of morphisms in $FJ$ are small relative to levelwise cofibrations follows from an adjunction argument, Definition \ref{Gcofgen} (ii) and \cite[2.1.16]{Hov99}. It remains to show that the domains of morphisms from $K$ are small relative to levelwise cofibrations. Any morphism in $K$ is a pushout-product of the form
$$\kappa_W \boxempty F_Vi \colon  (M\lambda_W \wedge F_VA) \bigvee_{F_WS^W \wedge F_VA} (F_WS^W \wedge F_VB) \lra M\lambda_W \wedge F_VB.$$
where the morphism $i \colon A \lra B$ is from the set $I$ and $V$ and $W$ are finite dimensional orthogonal $G$-representations. For any finite $G$-CW complex $L$ and any object $D$ which is the domain or codomain of a map from $I$, the spectrum $F_WL \wedge F_VD$ is small relative to levelwise cofibrations. Indeed, we have an isomorphism
$$\Hom(F_WL \wedge F_VD, X) \cong \Map(L \wedge D , X(V \oplus W)).$$
Since a pushout of small objects is small, Definition \ref{Gcofgen} (i) implies that $L \wedge D$ is small with respect to $I$-cof and hence $F_WL \wedge F_VD$ is small relative to levelwise cofibrations. Now we use twice that pushouts of small objects are small. First we conclude that $M\lambda_W \wedge F_VA$ is small relative to levelwise cofibrations and then we also see that
$$(M\lambda_W \wedge F_VA) \bigvee_{F_WS^W \wedge F_VA} (F_WS^W \wedge F_VB)$$
is small relative to levelwise cofibrations. \end{proof}

\subsection{$G$-equivariant stable model categories and completing the proof of Proposition \ref{spectrific}}

We start with the following

\begin{defi} \label{GSpec} An $\Sp^O_G$-model category is called \emph{$G$-spectral}. In other words, a model $\C$ category is $G$-spectral if it is  enriched, tensored and cotensored over the model category $\Sp^O_G$ and the pushout-out product axiom for tensors holds (see Definition~\ref{Vmod}).\end{defi}

By Proposition \ref{Gorthomon} the model category $\Sp^O_G$ is $G$-spectral. Next, Proposition \ref{stablespec} shows that the model structure of Proposition \ref{stmod} on $\Sp^O_G(\C)$ is $G$-spectral.

Recall from Definition \ref{eqmodel} that a $G$-equivariant stable model category is a $G \Top_*$-model category such that the Quillen adjunction
$$\xymatrix{ S^V \wedge -  \colon \C \ar@<0.5ex>[r] & \C \cocolon \Omega^V(-) \ar@<0.5ex>[l]}$$
is a Quillen equivalence for any finite dimensional orthogonal $G$-representation $V$. Before stating the next proposition, note that every $G$-spectral model category is obviously a $G \Top_*$-model category.

\begin{prop} \label{gspec->eq} Let $\C$ be a $G$-spectral model category. Then $\C$ is a $G$-equivariant stable model category. \end{prop}

\begin{proof} Consider the left Quillen functors
$$S^V \wedge - \colon \C \lra \C \;\; \text{and}\;\; F_VS^0 \wedge - \colon \C \lra \C$$
and their derived functors
$$S^V \wedge^{\mathbf{L}} - \colon \Ho(\C) \lra \Ho(\C) \;\; \text{and}\;\; F_VS^0 \wedge^{\mathbf{L}} -\colon \Ho(\C) \lra \Ho(\C).$$
Since the map $\lambda_V \colon F_VS^V \lra \S$ is a stable equivalence \cite[III.4.5]{MM02}, for every cofibrant $X$ in $\C$, one has the following weak equivalences
$$\xymatrix{S^V \wedge F_V S^0 \wedge X \cong F_VS^V \wedge X \ar[r]_-{\simeq}^-{\lambda_V \wedge 1} & X}$$
and
$$\xymatrix{F_VS^0 \wedge S^V \wedge X \cong F_VS^V \wedge X \ar[r]_-{\simeq}^-{\lambda_V \wedge 1} & X.}$$ 
This implies that the functors $S^V \wedge^{\mathbf{L}} - $ and $F_VS^0 \wedge^{\mathbf{L}} -$ are mutually inverse equivalences of categories. \end{proof}

\begin{coro} \label{SpGst} Let $\C$ be a proper and cofibrantly generated $G \Top_*$-model category which is stable as an ordinary model category. Then the category $\Sp^O_G(\C)$ together with the model structure of Proposition \ref{stmod} is a $G$-equivariant stable model category. 
\end{coro}

From this point on, the model structure of Proposition \ref{stmod} will be referred to as \emph{the stable model structure on $\Sp^O_G(\C)$} and the symbol $\Sp^O_G(\C)$ will always denote this model category. 

\;

\;

Finally, we observe that the $G \Top_*$-adjunction
$$\xymatrix{ F_0=\Sigma^{\infty}  \colon \C \ar@<0.5ex>[r] & { \Sp}^O_G(\C) \cocolon \Ev_0 \ar@<0.5ex>[l]}$$
is a Quillen equivalence for every cofibrantly generated (in the sense of Definition \ref{Gcofgen}) and proper $G$-equivariant stable model category $\C$. The proof of this fact is a verbatim translation of the last part of the proof of \cite[Theorem 3.8.2]{SS03} to our case. This finishes the proof of Proposition \ref{spectrific}.

\begin{remk} \label{remenrich} The Quillen equivalence
$$\xymatrix{\Sigma^{\infty}  \colon \C \ar@<0.5ex>[r] & { \Sp}^O_G(\C) \cocolon \Ev_0 \ar@<0.5ex>[l]}$$
is in fact a $G \Top_*$-Quillen equivalence. Indeed, $(\Sigma^{\infty}, \Ev_0)$ is a $G \Top_*$-enriched adjunction and an enriched adjunction which is an underlying Quillen equivalence is an enriched Quillen equivalence by definition. Next, since enriched left adjoints preserve tensors \cite[Sections 3.2 and 3.7]{Kel05}, the functor $\Sigma^{\infty}$ preserves tensors. Similarly, the right adjoint $\Ev_0$ preserves cotensors. Further, the equivalence
$$\xymatrix{\LL\Sigma^{\infty}  \colon \Ho(\C) \ar@<0.5ex>[r] & \Ho({ \Sp}^O_G(\C)) \cocolon \R\Ev_0 \ar@<0.5ex>[l]}$$
is a $\Ho(G \Top_*)$-enriched equivalence. Finally, we note that the functor $\LL\Sigma^{\infty}$ preserves derived tensors and since $\R \Ev_0$ is an inverse of $\LL\Sigma^{\infty}$, it is also compatible with derived tensors. \end{remk}

\subsection{The $p$-local model structure on $G$-equivariant orthogonal spectra}
This subsection reviews the $p$-localization of the stable model structure on $\Sp^O_G$ for any prime $p$. Note that one can construct the $p$-local model structure on $\Sp^O_G$ by using general localization techniques of \cite{Hir} or \cite{Bou01}. Another possibility is to translate the arguments of \cite[Section 4]{SS02} to the equivariant context.

\begin{defi} {\rm (i)}  A map $f \colon X \lra Y$ of orthogonal $G$-spectra is called a $p$-local equivalence if the induced map
$$\pi_*^H(f) \otimes \Z_{(p)} \colon \pi_*^HX \otimes \Z_{(p)} \lra \pi_*^HY \otimes \Z_{(p)}$$
is an isomorphism for any subgroup $H$ of $G$. 

{\rm (ii)} A map $f \colon X \lra Y$ of orthogonal $G$-spectra is called a $p$-local fibration if it has the right lifting property with respect to all maps that are cofibrations and $p$-local equivalences. \end{defi}

Recall from Section \ref{prel} that the stable model structure on $\Sp^O_G$ is cofibrantly generated with $I_{\st}^G = I_{\lv}^G$ and $J_{\st}^G = K^G \cup J_{\lv}^G$ generating cofibrations and acyclic cofibrations, respectively. Further, we also recall that the mod $l$ Moore space $M(l)$ is defined by the following pushout
$$\xymatrix{S^1 \ar[r]^{\cdot l} \ar[d] & S^1 \ar[d] \\ CS^1 \ar[r] & M(l).}$$
($C(-)= (I, 0) \wedge -$ is the pointed cone functor.) Let $\iota \colon M(l) \lra CM(l)$ denote the inclusion of $M(l)$ into the cone $CM(l)$. Define $J_{(p)}^G$ to be the set of maps of orthogonal $G$-spectra
$$F_n(G/H_{+} \wedge \Sigma^m \iota) \colon F_n(G/H_{+} \wedge \Sigma^m M(l)) \lra F_n(G/H_{+} \wedge \Sigma^m CM(l)),$$
where $n, m \geq 0$, $H \leq G$ and $l$ is prime to $p$, i.e., invertible in $\Z_{(p)}$. We let $J_{\loc}^G$ denote the union $J_{\st}^G \cup J_{(p)}^G$. 

\begin{prop} \label{p-localmod} Let $G$ be a finite group and $p$ a prime. Then the category $\Sp^O_G$ of $G$-equivariant orthogonal spectra together with $p$-local equivalences, cofibrations and $p$-local fibrations forms a cofibrantly generated stable model category. The set $I^{G}_{\st}$ generates cofibrations and the set $J_{\loc}^{G}$ generates acyclic cofibrations. Furthermore, the fibrant objects are precisely the $G$-$\Omega$-spectra whose $H$-equivariant homotopy groups are $p$-local for any $H \leq G$. \end{prop}

\begin{proof} We don't give the details here and only observe that it is completely analogous to the proof in \cite[Section 4]{SS02}. \end{proof}

\begin{prop} \label{p-locmonoid} Suppose that $i \colon K \lra L$ and $j \colon A \lra B$ are cofibrations in $\Sp_G^O,_{(p)}$. Then the pushout-product
$$i \boxempty j \colon K \wedge B \bigvee_{K \wedge A} L \wedge A \lra L \wedge B$$
is a cofibration in $\Sp_G^O,_{(p)}$. Moreover, if in addition $i$ or $j$  is a $p$-local equivalence (i.e., a weak equivalence in $\Sp_G^O,_{(p)}$), then so is $i \boxempty j$.
 
\end{prop}

\begin{proof} By \cite[Corollary 4.2.5]{Hov99} it suffices to check the claim for generating cofibrations and acyclic cofibrations. We don't give details because they are straightforward. \end{proof}

\;

\;

Since every stable equivalence of $G$-equivariant orthogonal spectra is a $p$-local equivalence, one obtains the following corollary:

\begin{coro} The model category $\Sp^O_G,_{(p)}$ is $G$-spectral, i.e., an $\Sp^O_G$-model category (see Definition \ref{GSpec}). \end{coro}

In view of Proposition \ref{gspec->eq}, we also obtain

\begin{coro} The model category $\Sp^O_G,_{(p)}$ is a $G$-equivariant stable model category (see Definition \ref{eqmodel}). \end{coro}

We end this subsection with some useful comments and remarks about the homotopy category $\Ho(\Sp^O_G,_{(p)})$. Since the model category $\Sp^O_G,_{(p)}$ is stable, the homotopy category $\Ho(\Sp^O_G,_{(p)})$ is naturally triangulated. Further, the set
$$\{\Sigma_{+}^{\infty} G/H \; | \; H \leq G\}$$
is a set of compact generators for $\Ho(\Sp^O_G,_{(p)})$. This follows from the natural isomorphism
$$[\Sigma_{+}^{\infty} G/H, X]^{\Ho(\Sp^O_G,_{(p)})}_* \cong \pi_*^H X \otimes \Z_{(p)}.$$
Finally, we note that for any $G$-equivariant orthogonal spectra $X$ and $Y$, the abelian group of morphisms $[X,Y]^{\Ho(\Sp^O_G,_{(p)})}$ in $\Ho(\Sp^O_G,_{(p)})$ (which will be also denoted by $[X,Y]^G$ abusing notation) is $p$-local. This follows from the fact that for any integer $l$ which is prime to $p$, the map $l \cdot \id \colon X \lra X$ is an isomorphism in $\Ho(\Sp^O_G,_{(p)})$.

\subsection{Reduction to Theorem \ref{reduction}} \label{p-locred}

Now we are finally ready to explain why the arguments of Subsection \ref{outl} carry over to the $p$-local case.

Let $\C$ be a cofibrantly generated (in the sense of Definition \ref{Gcofgen}), proper, $G$-equivariant stable model category. Suppose that 
$$\xymatrix{\Psi \colon \Ho( \Sp_G^O,_{(p)}) \ar[r]^-{\sim} & \Ho(\C)}$$
is an equivalence of triangulated categories such that 
$$\Psi(\Sigma_{+}^{\infty} G/H) \cong G/H_{+} \wedge^{\mathbf{L}} \Psi(\S),$$
for any $ H \leq G$. Suppose further that the latter isomorphisms are natural with respect to the restrictions, conjugations and transfers. By Proposition \ref{spectrific}, there is a $G \Top_*$-Quillen equivalence
$$\xymatrix{ \Sigma^{\infty}  \colon \C \ar@<0.5ex>[r] & { \Sp}^O_G(\C) \cocolon \Ev_0. \ar@<0.5ex>[l]}$$
Next, as in Subsection \ref{outl}, let $X$ be a cofibrant replacement of $(\mathbf{L}\Sigma^{\infty} \circ \Psi) (\S)$. Since $\Sp^O_G(\C)$ is $G$-spectral (Proposition \ref{stablespec}), there is a $G \Top_*$-Quillen adjunction 
$$\xymatrix{-\wedge X \colon \Sp_G^O \ar@<0.5ex>[r] & {\Sp}^O_G(\C) \cocolon \Hom(X,-).  \ar@<0.5ex>[l]}$$

\noindent Since the Hom groups of $\Ho(\C)$ are $p$-local, it is easy to see that the latter Quillen adjunction yields a $G \Top_*$-Quillen adjunction
$$\xymatrix{-\wedge X \colon \Sp^O_G,_{(p)} \ar@<0.5ex>[r] & {\Sp}^O_G(\C) \cocolon \Hom(X,-).  \ar@<0.5ex>[l]}$$
Next, choose $\Psi^{-1}$ as in Subsection \ref{outl} and consider the composite  
$$\xymatrix{ F \colon \Ho(\Sp^O_G,_{(p)}) \ar[r]^-{-\wedge^{\mathbf{L}} X} & \Ho(\Sp^O_G(\C)) \ar[r]^-{\mathbf{R}\Ev_0} & \Ho(\C) \ar[r]^-{\Psi^{-1}} & \Ho(\Sp^O_G,_{(p)}).}$$
Since the functors $\mathbf{R}\Ev_0$ and $\Psi^{-1}$ are equivalences, to prove that $(-\wedge X, \Hom(X,-))$ is a Quillen equivalence is equivalent to showing that the endofunctor 
$$\xymatrix{ F \colon \Ho(\Sp^O_G,_{(p)}) \ar[r] & \Ho(\Sp^O_G,_{(p)}) }$$ 
is an equivalence of categories. By the assumptions and the properties of $\Psi^{-1}$, we see that $F$ enjoys the following properties:

\;

\;

{\rm (i)} $F(\Sigma_{+}^{\infty} G/H) = \Sigma_{+}^{\infty} G/H, \;\; H \leq G$; 

\;

\;

{\rm (ii)} $F(g)=g, \;\; F(\res_{K}^H)= \res_{K}^H, \;\; F(\tr_{K}^H)=\tr_{K}^H, \;\; g \in G, \;\; K \leq H \leq G$;

\;

\;

{\rm (iii)} $F$ is an exact functor of triangulated categories and preserves infinite coproducts.

\;

\;

So finally, we see that in order to prove Theorem \ref{damtkicebuli}, it suffices to prove Theorem \ref{reduction}. Note that we do not expect that an odd primary version of Theorem \ref{reduction} is true. However, we still think that Conjecture \ref{mainconj} holds. Schwede's paper \cite{Sch07} suggests that the proof in the odd primary case should use the explicit construction of the endofunctor $F$, whereas in the $2$-local case certain axiomatic properties of $F$ are enough to get the desired result as Theorem \ref{reduction} shows. This is a generic difference between the $2$-local case and the $p$-local case for $p$ an odd prime.

\section{Free G-spectra} \label{free}

Since the set $\{\Sigma_{+}^{\infty} G/H \; | \; H \leq G\}$ is a set of compact generators for the triangulated category $\Ho(\Sp_G^O,_{(2)})$, to prove Theorem \ref{reduction} it suffices to show that for any subgroups $H$ and $K$ of $G$, the map
$$\xymatrix{F \colon [\Sigma_{+}^{\infty} G/H, \Sigma_{+}^{\infty} G/K]_*^G \ar[r] & [F(\Sigma_{+}^{\infty} G/H), F(\Sigma_{+}^{\infty} G/K)]_*^G = [\Sigma_{+}^{\infty} G/H, \Sigma_{+}^{\infty} G/K]_*^G} $$
induced by $F$ is an isomorphism.

\setcounter{subsection}{0}

In this section we show that under the assumptions of \ref{reduction} the map
$$F \colon [\Sigma_{+}^{\infty} G, \Sigma_{+}^{\infty} G]_*^G \lra [F(\Sigma_{+}^{\infty} G), F(\Sigma_{+}^{\infty} G)]_*^G = [\Sigma_{+}^{\infty} G, \Sigma_{+}^{\infty} G]_*^G$$ 
is an isomorphism. Note that the graded endomorphism ring $[\Sigma_{+}^{\infty} G, \Sigma_{+}^{\infty} G]_*^G$ is isomorphic to the graded group algebra $\pi_*\S[G]$ and the localizing subcategory generated by $\Sigma_{+}^{\infty} G$ in $\Ho(\Sp_G^O)$ is equivalent to $\Ho(\Mod \Sigma_{+}^{\infty} G)$, where $\Sigma_{+}^{\infty}G$ is considered as the group ring spectrum of $G$. 

\;

\;

We say that an object $X \in \Ho(\Sp^O_G)$ is a \emph{free $G$-spectrum} if it is contained in the localizing subcategory generated by $\Sigma_{+}^{\infty} G$. 

\;

\;

In the rest of the paper everything will be $2$-localized and hence we will mostly omit the subscript $2$.

\newpage

\subsection{Cellular structures} 

We start with the following

\begin{defi} \label{cell} Let $R$ be an orthogonal ring spectrum, $X$ an $R$-module and $n$ and $m$ integers such that $n \leq m$. We say that $X$ admits a \emph{finite $(n,m)$-cell structure} if there are sequences of distinguished triangles
$$\xymatrix{\bigvee_{I_k} \Sigma^{k-1}R \ar[r] & \sk_{k-1} X \ar[r] & \sk_k X \ar[r] & \bigvee_{I_k} \Sigma^k R }$$
in $\Ho(\Mod R)$, $k=n, n+1,...,m$, such that the sets $I_k$ are finite and $\sk_{n-1} X = \ast$ and $\sk_m X = X$.
 
\end{defi}

In other words, an $R$-module $X$ admits a finite $(n,m)$-\emph{cell structure} if and only if it admits a structure of a finite $R$-cell complex with all possible cells in dimensions between $n$ and $m$. 

Recall that there is a Quillen adjunction
$$\xymatrix{G_+ \wedge - \colon \Mod \S \ar@<0.5ex>[r] & \Mod \Sigma_{+}^{\infty} G \cocolon U  \ar@<0.5ex>[l]}$$
and that $[\Sigma_{+}^{\infty} G, \Sigma_{+}^{\infty} G]_*^G \cong \pi_*\S[G]$. The following proposition can be considered as a $2$-local naive equivariant version of \cite[Lemma 4.1]{Sch01} (cf. \cite[4.2]{Coh68}). 

\begin{prop} \label{factorization} Any $\alpha \in [\Sigma_{+}^{\infty} G, \Sigma_{+}^{\infty} G]_{n}^{\Ho(\Sp_G^O,_{(2)})}$, $n \geq 8$, factors over an $\Sigma_{+}^{\infty} G$-module that admits a finite $(1, n-1)$-cell structure. \end{prop}

\begin{proof} We will omit the subscript $2$. Under the derived adjunction
$$\xymatrix{G_+ \wedge^{\mathbf{L}} - \colon \Ho(\Mod \S) \ar@<0.5ex>[r] & \Ho(\Mod \Sigma_{+}^{\infty} G) \cocolon \mathbf{R}U,  \ar@<0.5ex>[l]}$$
the element $\alpha$ corresponds to some map $\wt{\alpha} \colon \S^n \lra \mathbf{R} U (\Sigma_{+}^{\infty} G) \cong \bigvee_{G} \S$. By the proof of \cite[Lemma 4.1]{Sch01}, for any $g \in G$, we have a factorization
$$\xymatrix{\S^n \ar[r]^-{\wt{\alpha}} \ar[dr] & \mathbf{R} U (\Sigma_{+}^{\infty} G) \cong \bigvee_{G} \S \ar[r]^-{\proj_g} & \S \\ & Z_g \ar[ur] & }$$
in the stable homotopy category, where $Z_g$ has $\S$-cells in dimensions between $1$ and $n-1$. This uses essentially that $n \geq 8$. Indeed, since $n \geq 8$, for any $g \in G$, the morphism $\proj_g \circ \wt{\alpha}$ has $\FF_2$-Adams filtration at least $2$ by the Hopf invariant one Theorem \cite{Ada60} and hence, one of the implications of \cite[Lemma 4.1]{Sch01} applies to $\proj_g \circ \wt{\alpha}$. Assembling these factorizations together, we get a commutative diagram
$$\xymatrix{\S^n \ar[rr]^-{\wt{\alpha}} \ar[dr] & & \mathbf{R} U (\Sigma_{+}^{\infty} G) \\ & \bigvee_{g \in G} Z_{g}. \ar[ur] & }$$
Finally, by adjunction, one obtains the desired factorization
$$\xymatrix{\Sigma^n\Sigma_{+}^{\infty} G \ar[rr]^-{\alpha} \ar[dr] & & \Sigma_{+}^{\infty} G \\ & G_{+}\wedge^{\mathbf{L}}(\bigvee_{g \in G} Z_{g}).  \ar[ur] & } \vspace{-0.7cm}$$  \qedhere 
\end{proof}

Next, we use Proposition \ref{factorization} to prove the following important

\begin{lem} \label{techlem} Suppose that the map of graded rings
$$F \colon [\Sigma_{+}^{\infty} G, \Sigma_{+}^{\infty} G]_*^G \lra [F(\Sigma_{+}^{\infty} G), F(\Sigma_{+}^{\infty} G)]_*^G = [\Sigma_{+}^{\infty} G, \Sigma_{+}^{\infty} G]_*^G$$ 
is an isomorphism below and including dimension $n$ for some $n \geq 0$. Then the following hold:

{\rm (i)} Let $K$ and $L$ be $\Sigma_{+}^{\infty} G$-modules that admit finite $(\beta_K, \tau_K)$ and $(\beta_L, \tau_L)$-cell structures, respectively, and assume that $\tau_K-\beta_L \leq n$. Then the map
$$F \colon [K,L]^G \lra [F(K), F(L)]^G$$
is an isomorphism.

{\rm (ii)} Let $K$ be an $\Sigma_{+}^{\infty} G$-module admitting a finite $(\beta_K, \tau_K)$-cell structure with $\tau_K-\beta_K \leq n+1$. Then there is an $\Sigma_{+}^{\infty} G$-module $K'$ with a finite $(\beta_{K'}, \tau_{K'})$-cell structure such that $\beta_K \leq \beta_{K'}$, $\tau_{K'} \leq  \tau_K$ and $F(K') \cong K$.

{\rm (iii)} If $n+1 \geq 8$, then the map
$$F \colon [\Sigma_{+}^{\infty} G, \Sigma_{+}^{\infty} G]_{n+1}^G \lra [\Sigma_{+}^{\infty} G, \Sigma_{+}^{\infty} G]_{n+1}^G$$
is an isomorphism.

\end{lem}

\begin{proof} {\rm (i)} When $K$ and $L$ are both finite wedges of type $\bigvee \Sigma^{l_0}\Sigma_{+}^{\infty} G$ ($l_0$ is fixed), then the claim holds. 

We start with the case when $L$ is a finite wedge of copies of $\Sigma^{l_0}\Sigma_{+}^{\infty} G$, for some integer $l_0$, and proceed by induction on $\tau_K- \beta_K$. As already noted, the claim holds when $\tau_K- \beta_K=0$. Now suppose we are given $K$ with $\tau_K- \beta_K=r$, $r \geq 1$, and assume that the claim holds for all $\Sigma_{+}^{\infty} G$-modules $M$ that have a finite $(\beta_M, \tau_M)$-cell structure with $\tau_M- \beta_M < r$. Consider the distinguished triangle
$$\xymatrix{\bigvee_{I_{\tau_K}} \Sigma^{\tau_K-1} \Sigma_{+}^{\infty} G \ar[r] & \sk_{\tau_K-1}K \ar[r] & K \ar[r] & \bigvee_{I_{\tau_K}} \Sigma^{\tau_K} \Sigma_{+}^{\infty} G.}$$
The $\Sigma_{+}^{\infty} G$-module $\sk_{\tau_K-1}K$ has a finite $(\beta_K, \tau_K-1)$-cell structure. For convenience, let $P$ denote the wedge $\bigvee_{I_{\tau_K}} \Sigma^{\tau_K-1} \Sigma_{+}^{\infty} G$. The latter distinguished triangle induces a commutative diagram
{\tiny $$\hspace{-0.2cm}\xymatrix{[\Sigma \sk_{\tau_K-1}K, L]^G \ar[r] \ar[d]^-F & [\Sigma P, L]^G \ar[r] \ar[d]^-F & [K,L]^G \ar[r] \ar[d]^-F & [\sk_{\tau_K-1}K,L]^G \ar[r] \ar[d]^-F & [P, L]^G \ar[d]^-F \\ [F(\Sigma \sk_{\tau_K-1}K), F(L)]^G \ar[r] & [F( \Sigma P), F(L)]^G \ar[r] & [F(K),F(L)]^G \ar[r] & [F(\sk_{\tau_K-1}K),F(L)]^G \ar[r] & [F(P), F(L)]^G}$$}
\hspace{-0.2cm}with exact rows (The functor $F$ is exact.). By the induction basis, the second and the last vertical morphisms in this diagram are isomorphisms. The fourth morphism is an isomorphism by the induction assumption. Finally, since $\Sigma \sk_{\tau_K-1}K$ has a finite $(\beta_K+1, \tau_K)$-cell structure, the first vertical map is also an isomorphism by the induction assumption. Hence, the claim follows by the Five lemma.

Next, we do a similar induction with respect to $\tau_L - \beta_L$. The case $\tau_L- \beta_L=0$ is the previous paragraph. For the inductive step we choose a distinguished triangle
$$\xymatrix{\sk_{\beta_L}L \ar[r] & L \ar[r] & L' \ar[r] & \Sigma \sk_{\beta_L} L.}$$
The octahedral axiom implies that the $\Sigma_{+}^{\infty} G$-module $L'$ admits a finite $(\beta_L+1, \tau_L)$-cell structure. Now, as in the previous case, a five lemma argument finishes the proof.

{\rm (ii)} We do induction on $\tau_K- \beta_K$. If $\tau_K- \beta_K=0$, then $K$ is stably equivalent to a finite wedge  $\bigvee \Sigma^{l_0}\Sigma_{+}^{\infty} G$, for a fixed integer $l_0$, and the claim holds since $F(\Sigma_{+}^{\infty} G) = \Sigma_{+}^{\infty} G$. For the induction step, choose a distinguished triangle 
$$\xymatrix{\bigvee_{I_{\tau_K}} \Sigma^{\tau_K-1} \Sigma_{+}^{\infty} G \ar[r]^-{\alpha} & \sk_{\tau_K-1}K \ar[r] & K \ar[r] & \bigvee_{I_{\tau_K}} \Sigma^{\tau_K} \Sigma_{+}^{\infty} G.}$$
as above. By the induction assumption, there is an $\Sigma_{+}^{\infty} G$-module $M$ with a finite $(\beta_M, \tau_M)$-cell structure such that $\beta_K \leq \beta_M$, $\tau_M \leq \tau_K-1$ and $F(M) \cong \sk_{\tau_K-1}K$. Consider the composite
$$\xymatrix{F( \bigvee_{I_{\tau_K}}\Sigma^{\tau_K-1} \Sigma_{+}^{\infty} G) \ar[r]^-{\cong} & \bigvee_{I_{\tau_K}}\Sigma^{\tau_K-1} \Sigma_{+}^{\infty} G \ar[r]^-{\alpha} & \sk_{\tau_K-1}K \ar[r]^-{\cong} & F(M). }$$
Since $\tau_K-1 -\beta_M \leq \tau_K-1-\beta_K \leq n$, part (i) yields that there exists 
$$\alpha' \in [\bigvee_{I_{\tau_K}}\Sigma^{\tau_K-1} \Sigma_{+}^{\infty} G, M]^G$$ 
such that $F(\alpha')$ equals the latter composition. Next, choose a distinguished triangle
$$\xymatrix{\bigvee_{I_{\tau_K}} \Sigma^{\tau_K-1} \Sigma_{+}^{\infty} G \ar[r]^-{\alpha'} & M \ar[r] & K' \ar[r] & \bigvee_{I_{\tau_K}} \Sigma^{\tau_K} \Sigma_{+}^{\infty} G.}$$
The $\Sigma_{+}^{\infty} G$-module $K'$ has a finite $(\beta_M, \tau_K)$-cell structure. On the other hand, since $F$ is exact, one of the axioms for triangulated categories implies that there is a morphism $K \lra F(K')$ which makes the diagram
$$\xymatrix{\bigvee_{I_{\tau_K}} \Sigma^{\tau_K-1} \Sigma_{+}^{\infty} G \ar[r]^-{\alpha} \ar[d]^-{\cong} & \sk_{\tau_K-1}K \ar[r] \ar[d]^-{\cong} & K \ar[r] \ar[d] & \bigvee_{I_{\tau_K}} \Sigma^{\tau_K} \Sigma_{+}^{\infty} G \ar[d]^-{\cong} \\ F(\bigvee_{I_{\tau_K}} \Sigma^{\tau_K-1} \Sigma_{+}^{\infty} G) \ar[r]^-{F(\alpha')} & F(M) \ar[r] & F(K') \ar[r] & F(\bigvee_{I_{\tau_K}} \Sigma^{\tau_K} \Sigma_{+}^{\infty} G).}$$
commute. Now another five lemma argument shows that in fact the map $K \lra F(K')$ is an isomorphism in $\Ho(\Mod \Sigma_{+}^{\infty} G)$ and thus the proof of part (ii) is completed.

(iii) By Proposition \ref{factorization}, any morphism $\alpha \in [\Sigma^{n+1}\Sigma_{+}^{\infty} G, \Sigma_{+}^{\infty} G]^G$ factors over some $\Sigma_{+}^{\infty} G$-module $K$ which has a finite $(1, n)$-cell structure. By part (ii), there exists an $\Sigma_{+}^{\infty} G$-module $K'$ admitting a finite $(\beta_{K'}, \tau_{K'})$-cell structure and such that $1 \leq \beta_{K'}$, $\tau_{K'} \leq n$ and $F(K') \cong K$. Hence we get a commutative diagram
$$\xymatrix{F(\Sigma^{n+1}\Sigma_{+}^{\infty} G)  \cong \Sigma^{n+1}\Sigma_{+}^{\infty} G \ar[dr] \ar[rr]^-{\alpha} & & \Sigma_{+}^{\infty} G =  F(\Sigma_{+}^{\infty} G)\\ & F(K'). \ar[ur] &}$$
Since $n+1 - \beta_{K'} \leq n+1-1 =n$ and $\tau_{K'} -0=\tau_{K'} \leq n$, part (i) implies that both maps in the latter factorization are in the image of $F$. Hence, the map $\alpha$ is also in the image of the functor $F$ yielding that
$$F \colon [\Sigma_{+}^{\infty} G, \Sigma_{+}^{\infty} G]_{n+1}^G \lra [\Sigma_{+}^{\infty} G, \Sigma_{+}^{\infty} G]_{n+1}^G$$
is surjective. As the source and target of this morphism are finite of the same cardinality, we conclude that it is an isomorphism. \end{proof}

\begin{coro} \label{redlow} Let $F$ be as in \ref{reduction}. If the morphism
$$F \colon [\Sigma_{+}^{\infty} G, \Sigma_{+}^{\infty} G]_*^G \lra [\Sigma_{+}^{\infty} G, \Sigma_{+}^{\infty} G]_*^G$$ 
is an isomorphism for $\ast \leq 7$, then the functor $F$ restricts to an equivalence on the full subcategory of free $G$-spectra.\end{coro}

\subsection{Taking care of the dimensions $\leq 7$} 

In this subsection we show that the map 
$$F \colon [\Sigma_{+}^{\infty} G, \Sigma_{+}^{\infty} G]_*^G \lra [F(\Sigma_{+}^{\infty} G), F(\Sigma_{+}^{\infty} G)]_*^G = [\Sigma_{+}^{\infty} G, \Sigma_{+}^{\infty} G]_*^G$$ 
is an isomorphism for $\ast \leq 7$. By Corollary \ref{redlow} this will imply that the functor $F$ restricts to an equivalence on the full subcategory of free $G$-spectra.

Recall that we have a preferred isomorphism $[\Sigma_{+}^{\infty} G, \Sigma_{+}^{\infty} G]_*^G \cong \pi_*\S[G]$. Since the functor $F$ restricts to the identity on the stable Burnside (orbit) category, $F(g)=g$ for any $g \in G$. On the other hand, the map $F \colon \pi_*\S[G] \lra \pi_*\S[G]$ is a ring homomorphism and thus we conclude that it is an isomorphism for $\ast=0$. Note that $\pi_*\S[G]$ is finite for $\ast > 0$ and the Hopf maps $\eta$, $\nu$ and $\sigma$ multiplicatively generate $\pi_{\ast \leq 7} \S$. Hence, it suffices to show that the Hopf maps (considered as elements of $\pi_*\S[G]$ via the unit map $\S \lra \Sigma_{+}^{\infty} G$) are in the image of $F$.

We start by showing that $F(\eta)=\eta$. Recall that the mod $2$ Moore spectrum $M(2)$ in the $2$-localized (non-equivariant) stable homotopy category is defined by the distinguished triangle
$$\xymatrix{\S \ar[r]^-{2} & \S \ar[r]^-{\iota} & M(2) \ar[r]^-{\partial} & \S^1}$$
and the map $2 \colon M(2) \lra M(2)$ factors as a composite
$$\xymatrix{ M(2) \ar[r]^-{\partial} & \S^1 \ar[r]^-{\eta} & \S \ar[r]^-{\iota} & M(2).}$$
Applying the functor $G_+ \wedge^{\mathbf{L}} - \colon \Ho(\Mod \S) \lra \Ho(\Mod \Sigma_{+}^{\infty} G)$ to the distinguished triangle gives a distinguished triangle
$$\xymatrix{\Sigma_{+}^{\infty} G \ar[r]^-{2} & \Sigma_{+}^{\infty} G \ar[r]^-{1 \wedge \iota} & G_+ \wedge M(2) \ar[r]^-{1 \wedge \partial} & \Sigma\Sigma_{+}^{\infty} G}$$
in $\Ho(\Mod \Sigma_{+}^{\infty} G)$. Further, the map $2 \colon G_+ \wedge M(2) \lra G_+ \wedge M(2)$ factors as
$$\xymatrix{ G_+ \wedge M(2) \ar[r]^-{1 \wedge \partial} & \Sigma \Sigma_{+}^{\infty} G \ar[r]^{\eta} & \Sigma_{+}^{\infty} G \ar[r]^-{1 \wedge \iota} & G_+ \wedge M(2).}$$
One of the axioms for triangulated categories implies that we can choose an isomorphism 
$$F(G_+ \wedge M(2)) \cong G_+ \wedge M(2)$$ 
so that the diagram
$$\xymatrix{\Sigma_{+}^{\infty} G \ar[r]^-{2} \ar@{=}[d] & \Sigma_{+}^{\infty} G \ar[r]^-{1 \wedge \iota} \ar@{=}[d] & G_+ \wedge M(2) \ar[r]^-{1 \wedge \partial} \ar[d]^{\cong} & \Sigma\Sigma_{+}^{\infty} G \ar[d]^{\cong} \\ F(\Sigma_{+}^{\infty} G) \ar[r]^-{2} & F(\Sigma_{+}^{\infty} G) \ar[r]^-{F(1 \wedge \iota)} & F(G_+ \wedge M(2)) \ar[r]^-{F(1 \wedge \partial)} & F(\Sigma\Sigma_{+}^{\infty} G).}$$
commutes. We fix the latter isomorphism once and for all and identify $F(G_+ \wedge M(2))$ with $G_+ \wedge M(2)$. Note that under this identification the morphisms $F(1 \wedge \iota)$ and $F(1 \wedge \partial)$ correspond to $1 \wedge \iota$ and $1 \wedge \partial$, respectively. Next, since $F(2)=2$ and $2=(1 \wedge \iota) \eta (1 \wedge \partial)$, one gets the identity
$$(1 \wedge \iota) F(\eta) (1 \wedge \partial)=2.$$
It is well known that the map $2 \colon M(2) \lra M(2)$ is non-zero (In fact, $[M(2), M(2)] \cong \Z/4$)(see e.g. \cite[Proposition 4]{Sch10}). Hence, $2 \colon G_+ \wedge M(2) \lra G_+ \wedge M(2)$ is non-zero as there is a preferred  ring isomorphism
$$[G_+ \wedge M(2), G_+ \wedge M(2)]_*^G \cong [M(2), M(2)]_* \otimes \Z[G].$$
Now it follows that $F(\eta) \neq 0$. Suppose $F(\eta) = \sum_{g \in A} \eta g$, where $A$ is a non-empty subset of $G$. We want to show that $A=\{1\}$. The identity $(1 \wedge \iota) F(\eta) (1 \wedge \partial)=2$ yields
$$2=(1 \wedge \iota) (\sum_{g \in A} \eta g) (1 \wedge \partial)= \sum_{g \in A} (1 \wedge \iota) \eta (1 \wedge \partial) g=\sum_{g \in A}2g.$$
Once again using the isomorphism $[G_+ \wedge M(2), G_+ \wedge M(2)]_*^G \cong [M(2), M(2)]_* \otimes \Z[G]$ and the fact that $2 \neq 0$, we conclude that $A=\{1\}$ and hence, $F(\eta)=\eta$.

Next, we show that $\nu$ is in the image of $F$. Let
$$F(\nu) = m \nu + \sum_{g \in G\setminus\{1\}}n_g g\nu.$$
Recall that 2-locally we have the identity (see e.g. \cite[14.1 (i)]{Tod62})
$$\eta^3 =4\nu.$$
Since $F(\eta)=\eta$, after applying $F$ to this identity one obtains
$$4\nu =\eta^3=F(\eta^3)=F(4 \nu) = 4m\nu + \sum_{g \in G\setminus\{1\}}4n_g g\nu.$$
As the element $\nu$ is a generator of the group $\pi_3\S_{(2)} \cong \Z/8$, we conclude that $m=2k+1$, for some $k \in \Z$, and for any $g \in G\setminus\{1\}$,  $n_g=2l_g$, $l_g \in \Z$. Hence
$$F(\nu) = (2k+1) \nu + \sum_{g \in G\setminus\{1\}}2l_g g\nu.$$
Using that $F(g)=g$, we also deduce that
$$F(g_0\nu)= (2k+1)g_0 \nu + \sum_{g \in G\setminus\{1\}}2l_g g_0g\nu,$$
for any fixed $g_0 \in G\setminus\{1\}$. Thus the image of $F$ in $\pi_3\S_{(2)}[G] \cong \bigoplus_{G} \Z/8$ is additively generated by the rows a $G \times G$-matrix of the form
\[ \left(
    \begin{array}{cccccccccccc}
      2k+1 &         &     &     & \\ 
            & 2k+1     &     & \huge{\even}    &  \\ 
           & \huge{\even}         & \ddots        &     &  \\
            &           &       & 2k+1   &  \\
            &          &     &       & 2k+1 \\ 
  \end{array}\right), \]
where each diagonal entry is equal to $2k+1$ and all the other entries are even. Since the determinant of this matrix is odd and hence a unit in $\Z/8$, the homomorphism $F \colon \pi_3\S_{(2)}[G] \lra \pi_3\S_{(2)}[G]$ is an isomorphism and hence the element $\nu$ is in the image of $F$.

Finally, it remains to show that $\sigma \in \pi_7\S \subset \pi_7\S[G]$ is in the image of $F$. Recall that $\sigma$ is a generator of $\pi_7\S_{(2)} \cong \Z/16$. We use the Toda bracket relation
$$8\sigma =\langle \nu, 8, \nu \rangle$$
(see e.g. \cite[5.13-14]{Tod62})) in $\pi_*\S_{(2)}$ that holds without indeterminacy as $\pi_4\S=0$. This implies that $$8\sigma =\langle \nu, 8, \nu \rangle$$ in $\pi_*\S[G]$. Now since $F$ is an exact functor, one obtains
$$8 F(\sigma) = \langle F(\nu), 8, F(\nu)\rangle.$$
Recall that 
$$F(\nu) = (2k+1) \nu + \sum_{g \in G\setminus\{1\}}2l_g g\nu.$$
Let $F(\sigma) = m \sigma + \sum_{g \in G\setminus\{1\}}n_g g\sigma$. By \cite[Theorem 1.3]{BM09} and the relation $16\sigma=0$, we get
$$8(m\sigma + \sum_{g \in G\setminus\{1\}}n_g g\sigma)= \langle(2k+1) \nu + \sum_{g \in G\setminus\{1\}}2l_g g\nu, 8, (2k+1) \nu + \sum_{g \in G\setminus\{1\}}2l_g g\nu \rangle = 8(2k+1)^2 \sigma.$$
Hence we see that $m$ is odd and the numbers $n_g$ are even. Now a similar argument as in the case of $\nu$ implies that $F \colon \pi_7\S[G] \lra \pi_7\S[G]$ is surjective and hence $\sigma$ is in the image of $F$.

By combining the results of this subsection with Corollary \ref{redlow} we conclude that under the assumptions of \ref{reduction}, the functor $F \colon \Ho(\Sp_{G, (2)}^O) \lra  \Ho(\Sp_{G, (2)} ^O)$ becomes an equivalence when restricted to the full subcategory of free $G$-spectra, or equivalently, when restricted to $\Ho((\Mod \Sigma_{+}^{\infty} G)_{(2)})$. In fact, we have proved the following more general
\begin{prop} \label{gfree} Let $G$ be any finite group and 
$$F \colon \Ho((\Mod \Sigma_{+}^{\infty} G)_{(2)}) \lra \Ho((\Mod \Sigma_{+}^{\infty} G)_{(2)})$$ 
an exact endofunctor which preserves arbitrary coproducts and such that 
$$F(\Sigma_{+}^{\infty} G) =\Sigma_{+}^{\infty} G,$$ 
and $F(g)=g$ in $[\Sigma_{+}^{\infty} G, \Sigma_{+}^{\infty} G]^G$ for any  $g \in G$. Then $F$ is an equivalence of categories. \end{prop}

\section{Reduction to endomorphisms} \label{redtrres}

In this section we will show that in order to prove Theorem \ref{reduction} (and hence Theorem \ref{damtkicebuli}), it suffices to check that for any subgroup $L \leq G$, the map of graded endomorphism rings
$$F \colon [\Sigma_{+}^{\infty} G/L, \Sigma_{+}^{\infty} G/L]_*^G \lra [F(\Sigma_{+}^{\infty} G/L), F(\Sigma_{+}^{\infty} G/L)]_*^G = [\Sigma_{+}^{\infty} G/L, \Sigma_{+}^{\infty} G/L]_*^G$$
is an isomorphism. 

\subsection{Formulation}

Let $G$ be a finite group and $H$ and $K$ subgroups of $G$. For the rest of this section we fix once and for all a set $\{g \}$ of double coset representatives for $K \setminus G/ H$. Recall that for any $g \in G$, the conjugated subgroup $gHg^{-1}$ is denoted by ${}^gH$. Further,  
$$\kappa_g \colon [\Sigma_{+}^{\infty} G/H, \Sigma_{+}^{\infty} G/K]_*^G \lra [\Sigma_{+}^{\infty} G/ ({}^g H \cap K), \Sigma_{+}^{\infty} G/({}^g H \cap K) ]_*^G$$ 
will stand for the map which is defined by the following commutative diagram:
$$\xymatrix{[\Sigma_{+}^{\infty} G/H, \Sigma_{+}^{\infty} G/K]_*^G  \ar[d]_-{g^*} \ar[r]^-{\kappa_g} & [\Sigma_{+}^{\infty} G/({}^g H \cap K), \Sigma_{+}^{\infty} G/({}^g H \cap K)]_*^G \\ [\Sigma_{+}^{\infty} G/{}^g H, \Sigma_{+}^{\infty} G/ K]_*^G \ar[r]^-{(\tr^K_{{}^g H \cap K})_*} & [\Sigma_{+}^{\infty} G/{}^g H, \Sigma_{+}^{\infty} G/({}^g H \cap K)]_*^G. \ar[u]_-{(\res^{{}^g H}_{{}^g H \cap K})^*} }$$
The aim of this section is to prove 

\begin{prop} \label{splinj} The map 
$$\xymatrix{[\Sigma_{+}^{\infty} G/H, \Sigma_{+}^{\infty} G/K]_*^G \ar[rrr]^-{(\kappa_g)_{[g] \in K \setminus G /H}} &&& \bigoplus_{[g] \in K \setminus G /H} [\Sigma_{+}^{\infty} G/({}^g H \cap K) , \Sigma_{+}^{\infty} G/({}^g H \cap K)]_*^G}$$
is a split monomorphism.
 
\end{prop}

The author suspects that this statement is known to the experts. However, since we were unable to find a reference, we decided to provide a detailed proof here. The proof is mainly based on the equivariant Spanier-Whitehead duality (\cite[III.2, V.9]{LMS}, \cite[XVI.7]{alaska}) and on a combinatorial analysis of certain pointed $G$-sets.  

Before starting to prove Proposition \ref{splinj}, we explain how it reduces the proof of Theorem \ref{reduction} to endomorphisms. Indeed, there is a commutative diagram

\vspace{-0.45cm}

$$\xymatrix{[\Sigma_{+}^{\infty} G/H, \Sigma_{+}^{\infty} G/K]_*^G  \ar[d]_-{g^*} \ar[r]^-F  & [\Sigma_{+}^{\infty} G/H, \Sigma_{+}^{\infty} G/K]_*^G  \ar[d]^-{g^*} \\ [\Sigma_{+}^{\infty} G/{}^g H, \Sigma_{+}^{\infty} G/ K]_*^G \ar[d]_-{(\tr^K_{{}^g H \cap K})_*} \ar[r]^-F & [\Sigma_{+}^{\infty} G/{}^g H, \Sigma_{+}^{\infty} G/ K]_*^G \ar[d]^-{(\tr^K_{{}^g H \cap K})_*} \\ [\Sigma_{+}^{\infty} G/{}^g H, \Sigma_{+}^{\infty} G/({}^g H \cap K)]_*^G \ar[d]_-{(\res^{{}^g H}_{{}^g H \cap K})^*} \ar[r]^-F & [\Sigma_{+}^{\infty} G/{}^g H, \Sigma_{+}^{\infty} G/({}^g H \cap K)]_*^G \ar[d]^-{(\res^{{}^g H}_{{}^g H \cap K})^*} \\ [\Sigma_{+}^{\infty} G/({}^g H \cap K), \Sigma_{+}^{\infty} G/({}^g H \cap K)]_*^G  \ar[r]^-F  & [\Sigma_{+}^{\infty} G/({}^g H \cap K), \Sigma_{+}^{\infty} G/({}^g H \cap K)]_*^G.}$$
for any $g \in G$, which implies that the diagram
$$\xymatrix{[\Sigma_{+}^{\infty} G/H, \Sigma_{+}^{\infty} G/K]_*^G  \ar[rr]^-{(\kappa_g)_{[g] \in K \setminus G /H}} \ar[d]_-F & & \bigoplus_{[g] \in K \setminus G /H} [\Sigma_{+}^{\infty} G/({}^g H \cap K) , \Sigma_{+}^{\infty} G/({}^g H \cap K)]_*^G  \ar[d]^-{ \bigoplus_{[g] \in K \setminus G /H} F} \\ [\Sigma_{+}^{\infty} G/H, \Sigma_{+}^{\infty} G/K]_*^G  \ar[rr]^-{(\kappa_g)_{[g] \in K \setminus G /H}}   & &  \bigoplus_{[g] \in K \setminus G /H} [\Sigma_{+}^{\infty} G/({}^g H \cap K) , \Sigma_{+}^{\infty} G/({}^g H \cap K)]_*^G }$$
commutes. If we now assume that for any subgroup $L \leq G$, the map
$$F \colon [\Sigma_{+}^{\infty} G/L, \Sigma_{+}^{\infty} G/L]_*^G \lra [\Sigma_{+}^{\infty} G/L, \Sigma_{+}^{\infty} G/L]_*^G$$
is an isomorphism, then the right vertical map in the latter commutative square is an isomorphism. Proposition \ref{splinj} implies that the horizontal maps are injective. Hence, by a simple diagram chase, it follows that the left vertical morphism is injective as well. But now we know that for $\ast=0$ the morphism 
$$F \colon [\Sigma_{+}^{\infty} G/H, \Sigma_{+}^{\infty} G/K]_*^G \lra [\Sigma_{+}^{\infty} G/H, \Sigma_{+}^{\infty} G/K]_*^G$$ 
is the identity and for $\ast >0$ it has the same finite source and target (Subsection \ref{hosp}). Combining this with the latter injectivity result allows us to conclude that the map
$$F \colon [\Sigma_{+}^{\infty} G/H, \Sigma_{+}^{\infty} G/K]_*^G \lra [\Sigma_{+}^{\infty} G/H, \Sigma_{+}^{\infty} G/K]_*^G$$
is indeed an isomorphism for any integer $\ast$. 

The rest of this section is devoted to the proof of Proposition \ref{splinj}.

\subsection{Induction and coinduction} \label{indcoind}

Let $H$ be a subgroup of $G$ and $ i \colon H \hookrightarrow G$ denote the inclusion. The class of all finite dimensional orthogonal $H$-representations of the form $i^*V$, where $V$ is a finite dimensional orthogonal $G$-representation, contains the trivial representation and is closed under direct sums. Hence, according to \cite[II.2.2, III.4.2]{MM02}, there is a stable model category $\Sp_{H \leq G}^O$ of $H$-equivariant orthogonal spectra indexed on the class of such representations (cf. Subsection \ref{Gorthostab}). Since finite dimensional orthogonal $H$-representations which come from $G$-representations are cofinal in the class of all finite dimensional orthogonal $H$-representations, \cite[V.1.10]{MM02} implies that the Quillen adjunction
$$\xymatrix{ \id  \colon \Sp_{H \leq G}^O \ar@<0.5ex>[r] & \Sp_H^O \cocolon \id \ar@<0.5ex>[l]}$$
is a Quillen equivalence. Next, recall that there is a Quillen adjunction
$$\xymatrix{ G \ltimes_{H} -  \colon \Sp_{H \leq G}^O \ar@<0.5ex>[r] & \Sp_G^O \cocolon \Res_H^G, \ar@<0.5ex>[l]}$$
where $(G \ltimes_{H}X)(V)= G_+ \wedge_{H} X(i^*V)$, for any $X \in \Sp_{H \leq G}^O$ and any finite dimensional orthogonal $G$-representation $V$. The functor $\Res_H^G$ is just the restriction along the map $i \colon H \hookrightarrow G$. In fact, the functor $\Res_H^G$ preserves weak equivalences and moreover, it is also a left Quillen functor as we see from the Quillen adjunction
$$\xymatrix{ \Res_H^G  \colon \Sp_G^O \ar@<0.5ex>[r] & \Sp_{H \leq G}^O  \cocolon \Map_H(G_+,-). \ar@<0.5ex>[l]}$$
The right adjoint $\Map_H(G_+,-)$ is defined by $\Map_H(G_+, X)(V) = \Map_H(G_+, X(i^*V))$. Now since the functor $\id \colon \Sp_{H \leq G}^O \lra \Sp_H^O$ is a left Quillen functor, we also get a Quillen adjunction
$$\xymatrix{ \Res_H^G  \colon \Sp_G^O \ar@<0.5ex>[r] & \Sp_H^O  \cocolon \Map_H(G_+,-). \ar@<0.5ex>[l]}$$
These Quillen adjunctions induce corresponding adjunctions on the derived level:
$$\xymatrix{ G \ltimes_{H} -  \colon \Ho(\Sp_H^O) \sim \Ho(\Sp_{H \leq G}^O) \ar@<0.5ex>[r] & \Ho(\Sp_G^O) \cocolon \Res_H^G, \ar@<0.5ex>[l]}$$
and
$$\xymatrix{ \Res_H^G  \colon \Ho(\Sp_G^O) \ar@<0.5ex>[r] & \Ho(\Sp_H^O)  \cocolon \Map_H(G_+,-). \ar@<0.5ex>[l]}$$
Here we slightly abuse notation by denoting point-set level functors and their associated derived functors with same symbols. Next, note that the equivalence 
$$\Ho(\Sp_H^O) \sim \Ho(\Sp_{H \leq G}^O)$$ 
is a preferred one and is induced from the Quillen equivalence at the very beginning of this subsection.

The adjunctions recalled here are in fact special instances of the ``change of groups'' and ``change of universe'' functors of \cite[V]{MM02}. The functor $G \ltimes_{H} -$ is usually called the \emph{induction} and the functor $\Map_H(G_+,-)$ is called the \emph{coinduction}.

Let now $G \ltimes_H X$ denote the balanced product $G_+ \wedge_H X$ for any pointed $G$-set (space) $X$. Consider the the following natural point-set level map:
$$w_H \colon G \ltimes_{H} X \lra \Map_H(G_+,X)$$
given by
\[ w_H([g,x])(\gamma) = \begin{cases} \gamma g x &\mbox{if } \gamma g \in H \\
\ast & \mbox{if } \gamma g \notin H. \end{cases}\]

We remind the reader of the following result due to Wirthm\"uller:

\begin{prop}[Wirthm\"uller Isomorphism, see e.g. \cite{may03}] \label{wirth} The map $w_H$ induces a natural isomorphism between the derived functors 
$$G \ltimes_{H} - \colon \Ho(\Sp_H^O) \lra \Ho(\Sp_G^O)$$ 
and 
$$\Map_H(G_+,-) \colon \Ho(\Sp_H^O) \lra \Ho(\Sp_G^O).$$ 
That is, the left and right adjoint functors of 
$$\Res_H^G \colon \Ho(\Sp_G^O) \lra \Ho(\Sp_H^O)$$
are naturally isomorphic. \end{prop}

As a consequence of the Wirthm\"uller isomorphism, one gets that for any subgroup $L \leq G$, the equivariant spectrum $\Sigma_{+}^{\infty} G/L$ is self-dual. Indeed, the map
$$\xymatrix{\Sigma_{+}^{\infty} G/L \cong G \ltimes_{L} \S \ar[r]^-{w_L} & \Map_L(G_+, \S) \cong \Map(\Sigma_{+}^{\infty} G/L, \S) \cong D(\Sigma_{+}^{\infty} G/L)}$$
is an isomorphism in $\Ho(\Sp_G^O)$, where $D$ is the equivariant Spanier-Whitehead duality functor (\cite[II.6, III.2, V.9]{LMS}, \cite[XVI.7]{alaska}). 

\;

\;

We conclude the subsection with the following well-known lemma and its corollaries. 

\begin{lem}[Double coset formula] \label{splitting} Suppose $G$ is a finite group and $H$ and $K$ arbitrary subgroups of $G$. Let $c_g \colon {}^gH \lra H$ denote the map $c_g(x)=g^{-1}xg$, $g \in G$. Then for any pointed $H$-set $X$, the $K$-equivariant maps
$$K \ltimes_{{}^g H \cap K } \Res^{{}^g H}_{{}^g H \cap K} (c_g^*X) \lra \Res^G_K (G \ltimes_H X), \;\;\;\;\; [k,x] \mapsto [kg, x],$$
induce a natural splitting
$$\bigvee_{[g] \in K \setminus G / H} K \ltimes_{{}^g H \cap K} \Res^{{}^g H}_{{}^g H \cap K}(c_g^*X) \cong \Res^G_K (G \ltimes_H X).$$ \end{lem} 

\begin{coro} \label{splitting1} Suppose $G$ is a finite group and $H$ and $K$ subgroups of $G$. Then for any $Y \in \Ho(\Sp_H^O)$, there is a natural splitting
$$\bigvee_{[g] \in K\setminus G /H} K \ltimes_{{}^gH  \cap K} \Res^{{}^gH}_{{}^gH  \cap K}(c_g^*Y) \cong \Res^G_K (G \ltimes_H Y).$$ \end{coro}

Note that if $X$ is a pointed $G$-set, then there is a natural isomorphism
$$G \ltimes_{H} \Res_H^G X \cong  G/H_+ \wedge X$$
given by $[g,x] \mapsto ([g] \wedge gx)$. 

\begin{coro} \label{orbisplit} The maps
$$G/({}^g H \cap K)_+ \lra G/H_+ \wedge G/K_+, \;\;\;\;\; [x] \mapsto [xg] \wedge [x]$$
of pointed $G$-sets induce a natural splitting
$$\xymatrix{\bigvee_{[g] \in K \setminus G / H} G/({}^g H \cap K)_+ \ar[r]^-{\cong} & G/H_+ \wedge G/K_+.}$$
\end{coro}

\begin{proof} By the last observation and Lemma \ref{splitting}, we have a chain of isomorphisms of pointed $G$-sets:
\begin{align*} \begin{split} \bigvee_{[g] \in K \setminus G / H} G/({}^g H \cap K)_+ \cong \bigvee_{[g] \in K \setminus G / H} G \ltimes_K (K/({}^g H \cap K))_+ \cong \hspace{1.7cm}  \\ G \ltimes_K (\bigvee_{[g] \in K \setminus G / H} K/({}^g H \cap K)_+ ) \cong G \ltimes_K (\bigvee_{[g] \in K \setminus G / H} K \ltimes_{{}^g H \cap K} S^0) \cong \hspace{1.3cm} \\ G \ltimes_{K}\Res_K^G(G \ltimes_H S^0) \cong G \ltimes_K \Res_K^G(G/H_+) \cong G/K_+ \wedge G/H_+ \cong G/H_+ \wedge G/K_+. \end{split} \end{align*} 
Here the last isomorphism is the twist. Going through these explicit isomorphisms we see that any $[x] \in G/({}^g H \cap K)_+$ is sent to $[xg] \wedge [x] \in G/H_+ \wedge G/K_+ $. \end{proof} 

\subsection{Proof of Proposition \ref{splinj}} As we already mentioned after Proposition \ref{wirth}, we have the isomorphisms
$$D(\Sigma_{+}^{\infty} G/L) \cong \Sigma_{+}^{\infty} G/L, \;\; L \leq G,$$ 
in $\Ho(\Sp_G^O)$, where $D$ is the equivariant Spanier-Whitehead duality. It follows from \cite[II.6, III.2, V.9]{LMS} (see also \cite{Lew96}) that under these isomorphisms the transfer maps correspond to restrictions. In particular, for any $g \in G$, the diagram
$$\xymatrix{D(\Sigma_{+}^{\infty} G/({}^g H \cap K)) \ar[rr]^-{D(\tr^K_{{}^g H \cap K})} & & D(\Sigma_{+}^{\infty} G/K)  \\ \Sigma_{+}^{\infty} G/({}^g H \cap K) \ar[rr]^-{\res^{K}_{{}^g H \cap K}} \ar[u]^{\cong} & & \Sigma_{+}^{\infty} G/K_. \ar[u]_{\cong}}$$
commutes. Combining this with the Spanier-Whitehead duality, for any $g \in G$, one gets the following commutative diagram with all vertical maps isomorphisms:
$$\xymatrix{[\Sigma_{+}^{\infty} G/{}^g H, \Sigma_{+}^{\infty} G/ K]_*^G \ar[d]_{\cong} \ar[rr]^-{(\tr^K_{{}^g H \cap K})_*} & & [\Sigma_{+}^{\infty} G/{{}^g H}, \Sigma_{+}^{\infty} G/({}^g H \cap K)]_*^G \ar[d]^{\cong} \\  [\Sigma_{+}^{\infty} G/{}^g H \wedge D(\Sigma_{+}^{\infty} G/ K), \S]_*^G \; \ar[rr]^-{(1 \wedge D(\tr^K_{{}^g H \cap K}))^*} \ar[d]_{\cong} & & [\Sigma_{+}^{\infty} G/{{}^g H} \wedge D(\Sigma_{+}^{\infty} G/({}^g H \cap K)), \S]_*^G \ar[d]^{\cong} \\ [\Sigma_{+}^{\infty} G/{}^g H \wedge \Sigma_{+}^{\infty} G/ K, \S]_*^G  \ar[d]_{\cong} \ar[rr]^-{(1 \wedge \res^{K}_{{}^g H \cap K})^*} & & [\Sigma_{+}^{\infty} G/{{}^g H} \wedge \Sigma_{+}^{\infty} G/({}^g H \cap K), \S]_*^G \ar[d]^{\cong} \\ [\Sigma^{\infty} (G/{}^g H_+ \wedge G/ K_+), \S]_*^G  \ar[rr]^-{(1 \wedge \res^{K}_{{}^g H \cap K})^*} & & [\Sigma^{\infty} (G/{}^g H_+ \wedge G/({}^g H \cap K)_+), \S]_*^G}$$
Using again the Spanier-Whitehead duality and that $\Sigma_{+}^{\infty} G/L$, $L \leq G$, is self-dual, we also have commutative diagrams for every $g \in G$:
$$\xymatrix{[\Sigma_{+}^{\infty} G/H, \Sigma_{+}^{\infty} G/K]_*^G   \ar[rr]^-{g^*} \ar[d]_{\cong} & & [\Sigma_{+}^{\infty} G/{}^g H, \Sigma_{+}^{\infty} G/ K]_*^G \ar[d]^{\cong} \\ [\Sigma_{+}^{\infty} G/H \wedge D(\Sigma_{+}^{\infty} G/K), \S ]_*^G \ar[d]_{\cong} \ar[rr]^-{(g \wedge 1)^*} & & [\Sigma_{+}^{\infty} G/{}^g H \wedge D(\Sigma_{+}^{\infty} G/ K), \S]_*^G \ar[d]^{\cong} \\ [\Sigma_{+}^{\infty} G/H \wedge \Sigma_{+}^{\infty} G/K, \S ]_*^G \ar[rr]^-{(g \wedge 1)^*}  \ar[d]_{\cong} & & [\Sigma_{+}^{\infty} G/{}^g H \wedge \Sigma_{+}^{\infty} G/ K, \S]_*^G \ar[d]^{\cong} \\ [\Sigma^{\infty} (G/H_+ \wedge G/K_+), \S ]_*^G \ar[rr]^-{(g \wedge 1)^*} & & [\Sigma^{\infty} (G/{}^g H_+ \wedge G/K_+), \S]_*^G}$$
and
$$\xymatrix{[\Sigma_{+}^{\infty} G/{}^g H, \Sigma_{+}^{\infty} G/({}^g H \cap K)]_*^G \ar[rr]^-{(\res^{{}^g H}_{{}^g H \cap K})^*} \ar[d]_{\cong} & & [\Sigma_{+}^{\infty} G/({}^g H \cap K), \Sigma_{+}^{\infty} G/({}^g H \cap K)]_*^G \ar[d]^{\cong} \\  [\Sigma_{+}^{\infty} G/{}^g H \wedge D(\Sigma_{+}^{\infty} G/({}^g H \cap K)), \S]_*^G \ar[rr]^-{(\res^{{}^g H}_{{}^g H \cap K} \wedge 1)^*} \ar[d]_{\cong}  & & [\Sigma_{+}^{\infty} G/({}^g H \cap K) \wedge  D(\Sigma_{+}^{\infty} G/({}^g H \cap K)), \S]_*^G \ar[d]^{\cong} \\  [\Sigma_{+}^{\infty} G/{}^g H \wedge \Sigma_{+}^{\infty} G/({}^g H \cap K), \S]_*^G \ar[rr]^-{(\res^{{}^g H}_{{}^g H \cap K} \wedge 1)^*} \ar[d]_{\cong}  & & [\Sigma_{+}^{\infty} G/({}^g H \cap K) \wedge  \Sigma_{+}^{\infty} G/({}^g H \cap K), \S]_*^G \ar[d]^{\cong} \\ [\Sigma^{\infty} (G/{}^g H_+ \wedge G/({}^g H \cap K)_+), \S]_*^G   \ar[rr]^-{(\res^{{}^g H}_{{}^g H \cap K} \wedge 1)^*} & & [\Sigma^{\infty} (G/({}^g H \cap K)_+ \wedge  G/({}^g H \cap K)_+), \S]_*^G.}$$
Hence by definition, for any $g \in G$, the morphism 
$$\kappa_g \colon [\Sigma_{+}^{\infty} G/H, \Sigma_{+}^{\infty} G/K]_*^G \lra [\Sigma_{+}^{\infty} G/ ({}^g H \cap K), \Sigma_{+}^{\infty} G/({}^g H \cap K) ]_*^G$$
is isomorphic to the morphism induced by the composite
$$\xymatrix{G/({}^g H \cap K)_+ \wedge  G/({}^g H \cap K)_+ \ar[d]_-{\res^{{}^g H}_{{}^g H \cap K} \wedge 1} \ar[rr] & & G/H_+ \wedge G/K_+ \\ G/{}^g H_+ \wedge G/({}^g H \cap K)_+  \ar[rr]^-{1 \wedge \res^{K}_{{}^g H \cap K}}  & &  G/{}^g H_+ \wedge G/K_+ \ar[u]_-{g \wedge 1} }$$
after applying the functor $[\Sigma^{\infty}(-), \S]^G_*$. To simplify notations let us denote this composite of maps of pointed $G$-sets by $\theta_g \colon G/({}^g H \cap K)_+ \wedge  G/({}^g H \cap K)_+ \lra  G/H_+ \wedge G/K_+ $. Thus, in order to prove Proposition \ref{splinj}, it suffices to check that the map of pointed $G$-sets
$$\xymatrix{\bigvee_{[g] \in K \setminus G / H } (G/({}^g H \cap K)_+ \wedge  G/({}^g H \cap K)_+)  \ar[rr]^-{(\theta_g)_{[g] \in K \setminus G / H}} & & G/H_+ \wedge G/K_+   }$$
has a $G$-equivariant section. This follows from the commutative diagram of pointed $G$-sets
$$\xymatrix{\bigvee_{[g] \in K \setminus G / H } (G/({}^g H \cap K)_+ \wedge  G/({}^g H \cap K)_+) \ar[rr]^-{(\theta_g)_{[g] \in K \setminus G / H}} & & G/H_+ \wedge G/K_+ \\ & & \bigvee_{[g] \in K \setminus G / H} G/({}^g H \cap K)_+, \ar[ull]^-{\bigvee_{[g] \in K \setminus G / H} \Delta_g \;\;\;} \ar[u]^-{\cong}}$$
where the vertical map is the isomorphism from Corollary \ref{orbisplit} and 
$$\Delta_g \colon G/({}^g H \cap K)_+ \lra G/({}^g H \cap K)_+ \wedge  G/({}^g H \cap K)_+$$ 
is the diagonal defined by $[x] \mapsto [x] \wedge [x]$ for any $g$. \qed

\section{A short exact sequence} \label{geofixsec}

This section constructs a split short exact sequence that will play a fundamental role in the inductive proof of Theorem \ref{reduction}. The author thinks that this short exact sequence is well-known to the experts. However, since we were unable to find a reference, we decided to provide a detailed proof here.

\subsection{Geometric fixed points and the inflation functor}

To construct the desired short exact sequence we need the geometric fixed point functor 
$$\Phi^N \colon \Sp_G^O \lra ~\Sp_J^O$$ 
associated to an extension of finite groups 
$$\xymatrix{E \colon 1 \ar[r] & N \ar[r]^{\iota} & G \ar[r]^{\varepsilon} & J \ar[r] & 1.}$$
This functor is constructed in \cite[V.4]{MM02} and has some useful properties. In particular the following holds:

\vspace{-0.08cm}
\begin{prop}[{\cite[V.4.5]{MM02}}] \label{geofixprop} Let $V$ be a finite dimensional orthogonal $G$-representation and $A$ a pointed $G$-space. Then there is a natural isomorphism of $J$-spectra
$$\Phi^N(F_VA) \cong F_{V^N} A^N.$$
Furthermore, the functor $\Phi^N \colon \Sp_G^O \lra \Sp_J^O$ preserves cofibrations and acyclic cofibrations. \end{prop}

\begin{coro} \label{geofixsusp} For any based $G$-space $A$, there is a natural isomorphism of $J$-spectra
$$\Phi^N(\Sigma^{\infty} A) \cong \Sigma^{\infty} (A^N).$$\end{coro}

Next, we will also need the inflation (change of scalars) functor $\varepsilon^* \colon \Sp^O_J \lra \Sp^O_G$ associated to an extension of finite groups
$$\xymatrix{E \colon 1 \ar[r] & N \ar[r]^{\iota} & G \ar[r]^{\varepsilon} & J \ar[r] & 1.}$$
This functor is a left adjoint to the point-set level (categorical) fixed point functor. In fact it is a left Quillen functor.    

By Proposition \ref{geofixprop} and Ken Brown's Lemma one can derive the functor $\Phi^N$ and get the functor
$$\Phi^N \colon \Ho(\Sp^O_G) \lra \Ho(\Sp^O_J).$$
We can also derive the left Quillen functor $\varepsilon^*$ and obtain the derived inflation
$$\varepsilon^* \colon \Ho(\Sp^O_J) \lra \Ho(\Sp^O_G).$$

The following proposition follows from \cite[II.9.10]{LMS} and \cite[VI.3-5]{MM02}.

\begin{prop}\label{geopbk} There is a triangulated natural isomoprhism
$$\id \cong \Phi^N \varepsilon^*$$
of endofunctors on $\Ho(\Sp^O_J)$.
\end{prop}

\subsection{Weyl groups} \label{Weyl1} Let $G$ be a finite group and $H$ a subgroup of $G$. Then $H$ is a normal subgroup of its normalizer $N(H)= \{g \in G \; | \; gH=Hg \}$ and the quotient group $W(H) = N(H)/H$ is called the \emph{Weyl group} of $H$. According to the previous subsection, the short exact sequence
$$\xymatrix{1 \ar[r] & H \ar[r]^-{\iota} & N(H) \ar[r]^-{\varepsilon} & W(H) \ar[r] & 1}$$
gives us the geometric fixed point functor
$$\Phi^H \colon \Ho(\Sp_{N(H)}^O)  \lra \Ho(\Sp_{W(H)}^O)$$
and the inflation functor
$$\varepsilon^* \colon \Ho(\Sp_{W(H)}^O) \lra \Ho(\Sp_{N(H)}^O).$$
By a slight abuse of notation, we will denote the composite functor 
$$\Phi^H \circ \Res^G_{N(H)} \colon \Ho(\Sp_G^O) \lra \Ho(\Sp_{W(H)}^O)$$ 
also by $\Phi^H$. It then follows from Corollary \ref{geofixsusp} that there is an isomorphism
$$\Phi^H(\Sigma_{+}^{\infty} G/H) \cong \Sigma^{\infty}_{+} (G/H)^H$$ 
in $\Ho(\Sp_{W(H)}^O)$. (This holds already on the point-set level.) Since $(G/H)^H = W(H)$ as $W(H)$-sets, one in fact gets an isomorphism
$$\Phi^H(\Sigma_{+}^{\infty} G/H) \cong \Sigma^{\infty}_{+} W(H)$$
in $\Ho(\Sp_{W(H)}^O)$. Further, by definition, one has $\varepsilon^*(\Sigma^{\infty}_{+} W(H)) = \Sigma_{+}^{\infty}N(H)/H$ in $\Ho(\Sp_{N(H)}^O)$ and hence we get 
$$G \ltimes_{N(H)} \varepsilon^* (\Sigma^{\infty}_{+} W(H)) = \Sigma_{+}^{\infty} G/H.$$
Having in mind these identifications, we are now ready to formulate the following immediate consequence of Proposition \ref{geopbk}.

\begin{prop} \label{geofixweyl} The composite
$$\hspace{-0.4cm} \xymatrix{[\Sigma^{\infty}_{+} W(H), \Sigma^{\infty}_{+} W(H)]_*^{W(H)} \ar[rr]^-{G \ltimes_{N(H)} \varepsilon^*} & & [\Sigma_{+}^{\infty} G/H, \Sigma_{+}^{\infty} G/H]^G_* \ar[r]^-{\Phi^H} &  [\Sigma^{\infty}_{+} W(H), \Sigma^{\infty}_{+} W(H)]_*^{W(H)}}$$
is an isomorphism. \end{prop}

\subsection{The short exact sequence} \label{sec1}

Suppose $G$ is a finite group and $\F$ a set of subgroups of $G$. The set $\F$ is said to be a \emph{family of subgroups of $G$} if it is closed under conjugation and taking subgroups.

\;

\;

Recall that for any finite group $G$ and any family $\F$, there is a classifying space ($G$-CW complex) $E\F$ characterized up to $G$-homotopy equivalence by the property that $E\F^H$ is contractible if $H \in \F$ and $E\F^H= \emptyset$ if $H \notin \F$ (see e.g. \cite{Elm83}). 

\;

\;

Let $\P$ denote the family of proper subgroups of $G$.  Consider the equivariant map $\xymatrix{E \P_+ \ar[r]^-{\proj} & S^0}$ which sends the elements of $E\P$ to the non-base point of $S^0$. The mapping cone sequence of this map (called the isotropy separation sequence) combined with the tom Dieck splitting \cite[II.7.7]{tom} gives the following well-known fact:

\begin{prop} \label{tomexact1} Suppose $G$ is a finite group. Then there is a split short exact sequence
$$\xymatrix{0 \ar[r] & [\S, \Sigma_+^{\infty} E \P]_*^G \ar[r]^-{\proj_*} & [\S, \S]^G_* \ar[r]^-{\Phi^G} & [\S, \S]_*  \ar[r] & 0.}$$ \end{prop}

Now suppose $H$ is a subgroup of $G$. Then for any $X \in \Ho(\Sp_G^O)$, there is a natural isomorphism
$$G \ltimes_{H} \Res_H^G X \cong  G/H_+ \wedge^{\mathbf{L}} X$$
given on the point-set level by $[g,x] \mapsto ([g] \wedge gx)$. In particular, 
$$G \ltimes_{H} \S \cong \Sigma_{+}^{\infty} G/H.$$
Having in mind this preferred isomorphism, we will once and for all identify $G \ltimes_{H} \S$ with $\Sigma_{+}^{\infty} G/H$. Next, let $\P[H]$ denote the family of proper subgroups of $H$. Note that this is a family with respect to $H$ and not necessarily with respect to the whole group $G$. Here is the main result of this section which is an important tool in the proof of Theorem \ref{reduction}:
\begin{prop} \label{ses} Let $G$ be a finite group and $H$ a subgroup. Then there is a split short exact sequence 
$$\hspace{-0.4cm}\xymatrix{[\Sigma_+^{\infty} G/H, G \ltimes_{H} \Sigma_+^{\infty} E \P[H]]_*^G \; \ar@{>->}[r]^-{\proj_*} & [\Sigma_+^{\infty} G/H, \Sigma_+^{\infty} G/H]_*^G \ar@{->>}[r]^-{\Phi^H} & [\Sigma_+^{\infty} W(H), \Sigma_+^{\infty} W(H)]_*^{W(H)}, }$$
where the morphism $\proj \colon G \ltimes_{H} \Sigma_+^{\infty} E \P[H] \lra \Sigma_+^{\infty} G/H$ is defined as the composite
$$\xymatrix{G \ltimes_{H} \Sigma_+^{\infty} E \P[H] \ar[rr]^-{G \ltimes_{H} \proj} & & G \ltimes_{H} \S \cong \Sigma_+^{\infty} G/H.}$$\end{prop}

Before proving this proposition we have to recall some important technical facts.
\subsection{Technical preparation} It follows immediately from the definition that for any $K \in \P[H]$, the set of $H$-fixed points of $G/K$ is empty. Together with Corollary \ref{geofixsusp} this implies that $\Phi^H(\Sigma_{+}^{\infty} G/K) = \ast$ in $\Ho(\Sp_{W(H)}^O)$. Since the classifying space $E\P[H]$ is built out of $H$-cells of orbit type $H/K$ with $K \leq H$ and $K \neq H$ one obtains:

\begin{prop} \label{geoprop1}  Let $G$ be a finite group. For any subgroup $H \leq G$, the $G$-CW complex $G \times_{H} E\P[H]$ is built out of $G$-cells of orbit type $G/K$ with $K \leq H$ and $K \neq H$. Furthermore, $\Phi^H( G \ltimes_{H} \Sigma_+^{\infty} E\P[H]) = \ast$ in $\Ho(\Sp_{W(H)}^O)$. \end{prop}

Next, we recall that the functor $\Map_H(G_+,-) \colon \Ho(\Sp^O_H) \lra \Ho(\Sp^O_G)$ is right adjoint to $\Res_H^G$. Recall also the map $w_H \colon G \ltimes_{H}(-) \lra \Map_H(G_+,-)$ inducing the Wirthm\"uller isomorphism (Proposition \ref{wirth}). The following proposition is an immediate consequence of the Wirthm\"uller isomorphism.

\begin{prop} \label{wirth1} For any $Y$ in $\Ho (\Sp_{H}^O)$, the natural map
$$\xymatrix{[\S, G \ltimes_{H} Y]^G_* \ar[r]^-{{w_H}_*} &  [\S, \Map_H(G_{+}, Y)]^G_* \cong [\Res^G_H(\S), Y]^H_* = [\S,Y]^H_*}$$ 
is an isomorphism. \end{prop} 

\begin{coro} \label{wirth2} Let $G$ be a finite group and $H$ and $K$ subgroups of $G$. Then for any spectrum $Y \in \Ho(\Sp_{H}^O)$, there is a natural isomorphism
$$[\Sigma_+^{\infty} G/K, G \ltimes_H Y]_*^G \cong \bigoplus_{[g] \in K \setminus G/H } [\S, \Res_{K \cap {}^g H}^{{}^g H} (c_g^*(Y))]_*^{K \cap {}^g H}.$$ \end{coro}

\begin{proof} By adjunction, Corollary \ref{splitting1} and Proposition \ref{wirth1}, one has the following chain of isomorphisms:
\begin{align*} \begin{split} [\Sigma_+^{\infty} G/K, G \ltimes_H Y]_*^G \cong [G \ltimes_K \S, G \ltimes_H Y]_*^G \cong [\S, \Res_K^G(G \ltimes_H Y)]_*^K \cong \hspace{1.5cm} \\ [\S, \bigvee_{[g] \in K \setminus G/H } K \ltimes_{K \cap {}^g H} \Res_{K \cap {}^g H}^{{}^g H} (c_g^*(Y))]_*^K  \cong \bigoplus_{[g] \in K \setminus G/H} [\S, K \ltimes_{K \cap {}^g H} \Res_{K \cap {}^g H}^{{}^g H} (c_g^*(Y))]_*^K  \cong \\ \bigoplus_{[g] \in K \setminus G/H } [\S, \Res_{K \cap {}^g H}^{{}^g H} (c_g^*(Y))]_*^{K \cap {}^g H}. \hspace{4cm}  \end{split}  \end{align*} \end{proof}

\subsection{Proof of Proposition \ref{ses}} It follows from Proposition \ref{geofixweyl} that 
$$\Phi^H \colon [\Sigma_+^{\infty} G/H, \Sigma_+^{\infty} G/H]_*^G \lra [\Sigma_+^{\infty} W(H), \Sigma_+^{\infty} W(H)]_*^{W(H)}$$
is a retraction and thus in particular surjective. Further, Proposition \ref{geoprop1} implies that 
$$\Phi^H \circ \proj_* =0.$$ 
Hence, it remains to show that the map
$$ \proj_* \colon [\Sigma_+^{\infty} G/H, G \ltimes_{H} \Sigma_+^{\infty} E \P[H]]_*^G \lra  [\Sigma_+^{\infty} G/H, \Sigma_+^{\infty} G/H]_*^G $$
is injective and $\Ker \Phi^H \subset \Imm (\proj_*)$. For this we choose a set $\{g \}$ of double coset representatives for $H \setminus G /H$. By Corollary \ref{wirth2}, there is a commutative diagram with all vertical arrows isomorphisms
$$\xymatrix{[\Sigma_+^{\infty} G/H, G \ltimes_{H} \Sigma_+^{\infty} E \P[H]]_*^G \ar[d]_{=} \ar[rrr]^-{\proj_*} & & &  [\Sigma_+^{\infty} G/H, \Sigma_+^{\infty} G/H]_*^G \ar[d]^{\cong} \\ [\Sigma_+^{\infty} G/H, G \ltimes_{H} \Sigma_+^{\infty} E \P[H]]_*^G \ar[d]_{\cong} \ar[rrr]^-{(G \ltimes_H \proj)_*} & & & [\Sigma_+^{\infty} G/H, G \ltimes_{H} \S]_*^G    \ar[d]^{\cong} \\ \bigoplus_{[g] \in H \setminus G /H}  [\S, \Sigma_+^{\infty} \Res_{H \cap {}^g H}^{{}^g H} (c_g^*(E \P[H]))]_*^{H \cap {}^g H} \ar[rrr]^-{\bigoplus\limits_{[g] \in H \setminus G /H}  (\proj)_*} & & & \bigoplus_{[g] \in H \setminus G /H}  [\S, \S ]_*^{H \cap {}^g H}. } $$
%

\noindent We will now identify the summands of the lower horizontal map. For this one has to consider two cases:

\textbf{Case 1. $H \cap {}^g H=H$}: In this case $\Res_{H \cap {}^g H}^{{}^g H} (c_g^*(E \P[H])) = c_g^*(E \P[H])$ is a model for classifying space of $\P[H]$ (and hence $G$-homotopy equivalent to $E \P[H])$). By Proposition \ref{tomexact1}, we get a short exact sequence
$$\xymatrix{0 \ar[r] & [\S, \Sigma_+^{\infty} \Res_{H \cap {}^g H}^{{}^g H} (c_g^*(E \P[H]))]_*^{H \cap {}^g H} \ar[r]^-{\proj_*} & [\S, \S ]_*^{H \cap {}^g H} \ar[r]^-{\Phi^{H \cap {}^g H}} & [\S, \S]_* \ar[r] & 0.}$$

\textbf{Case 2. $H \cap {}^g H$ is a proper subgroup of $H$}: In this case $\Res_{H \cap {}^g H}^{{}^g H} (c_g^*(E \P[H]))$ is an $(H \cap {}^g H)$-contractible cofibrant $(H \cap {}^g H)$-space and hence the map
$$\xymatrix{[\S, \Sigma_+^{\infty} \Res_{H \cap {}^g H}^{{}^g H} (c_g^*(E \P[H]))]_*^{H \cap {}^g H} \ar[r]^-{\proj_*} & [\S, \S ]_*^{H \cap {}^g H}}$$
is an isomorphism.

Altogether, after combining the latter diagram with Case 1. and Case 2., we see that the map
$$ \proj_* \colon [\Sigma_+^{\infty} G/H, G \ltimes_{H} \Sigma_+^{\infty} E \P[H]]_*^G \lra  [\Sigma_+^{\infty} G/H, \Sigma_+^{\infty} G/H]_*^G $$ is injective. It still remains to check that $\Ker \Phi^H \subset \Imm (\proj_*)$. For this, first note that $H \cap {}^g H=H$ if and only if $g \in N(H)$. Further, if $g \in N(H)$, then the double coset class $HgH$ is equal to $gH$. Hence, the set of those double cosets $[g] \in H \setminus G /H$ for which the equality $H \cap {}^g H=H$ holds is in bijection with the Weyl group $W(H)$. Consequently, using the latter diagram and Case 1. and Case 2., one gets an isomorphism
$$[\Sigma_+^{\infty} G/H, \Sigma_+^{\infty} G/H]_*^G/ \Imm (\proj_*) \cong \bigoplus_{W(H)}[\S, \S]_* \cong  [\Sigma_+^{\infty} W(H), \Sigma_+^{\infty} W(H)]_*^{W(H)}.$$
On the other hand, we have already checked that
$$\Phi^H \colon [\Sigma_+^{\infty} G/H, \Sigma_+^{\infty} G/H]_*^G \lra [\Sigma_+^{\infty} W(H), \Sigma_+^{\infty} W(H)]_*^{W(H)}$$
is surjective and this yields an isomorphism
$$[\Sigma_+^{\infty} G/H, \Sigma_+^{\infty} G/H]_*^G/\Ker \Phi^H \cong [\Sigma_+^{\infty} W(H), \Sigma_+^{\infty} W(H)]_*^{W(H)}.$$
Combining this with the previous isomorphism implies that the graded abelian group 
$$[\Sigma_+^{\infty} G/H, \Sigma_+^{\infty} G/H]_*^G/ \Imm (\proj_*)$$ 
is isomorphic to 
$$[\Sigma_+^{\infty} G/H, \Sigma_+^{\infty} G/H]_*^G/\Ker \Phi^H.$$ 
Now if the grading $\ast >0$, then $[\Sigma_+^{\infty} G/H, \Sigma_+^{\infty} G/H]_*^G$ is finite and it follows that $\Imm (\proj_*)$ and $\Ker \Phi^H$ are finite groups of the same cardinality (Subsection \ref{hosp}). Since we already know that  $\Imm(\proj_*) \subset \Ker \Phi^H$ (We have already observed that this is a consequence of Proposition \ref{geoprop1}.), one finally gets the equality $\Imm(\proj_*) = \Ker \Phi^H$. For $\ast=0$ a Five lemma argument completes the proof. We do not give here the details of the case $\ast =0$ as it is irrelevant for our proof of Theorem \ref{reduction}. $\Box$

\section{Proof of the main theorem} \label{prfmain} 

In this section we complete the proof of Theorem \ref{reduction} and hence of Theorem \ref{damtkicebuli}.

\;

\;

We start by recalling from \cite[IV.6]{MM02} the $\F$-model structure on the category of $G$-equivariant orthogonal spectra, where $\F$ is a family of subgroups of a finite group $G$.

\subsection{The $\F$-model structure and localizing subcategory determined by $\F$} \label{Fmodel} Let $G$ be a finite group and $\F$ a family of subgroups of $G$. 

\begin{defi} A morphism $f \colon X \lra Y$ of $G$-equivariant orthogonal spectra is called an $\F$-\emph{equivalence} if it induces isomorphisms
$$\xymatrix{f_* \colon \pi_*^H X \ar[r]^{\cong} & \pi_*^H Y}$$
on $H$-equivariant homotopy groups for any $H \in \F$. Similarly, a morphism $g \colon X \lra Y$ in $\Ho(\Sp^O_G)$ is called an $\F$-\emph{equivalence} if it induces an isomorphism on $\pi_*^H$ for any $H \in \F$.\end{defi}

The category of $G$-equivariant orthogonal spectra has a stable model structure with weak equivalences the $\F$-equivalences and with cofibrations the $\F$-cofibrations \cite[IV.6.5]{MM02}. By restricting our attention to those orbits $G/H$ which satisfy $H \in \F$, we can obtain the generating $\F$-cofibrations and acyclic $\F$-cofibrations in a similar way as for the absolute case of $\Sp^O_G$ \cite[III.4]{MM02} (see Subsection \ref{Gortholev} and Subsection \ref{Gorthostab}). We will denote this model category by $\Sp^O_{G, \F}$.
%
%
%
%

Any $\F$-equivalence can be detected in terms of geometric fixed points. To see this we need the following proposition which relates the classifying space $E\F$ with the concept of an $\F$-equivalence:

\begin{prop}[{\cite[IV.6.7]{MM02}}] \label{Fequi1} A morphism $f \colon X \lra Y$ of $G$-equivariant orthogonal spectra is an $\F$-equivalence if and only if $1 \wedge f \colon E\F_+ \wedge X \lra E\F_+ \wedge Y$ is a $G$-equivalence, i.e., a stable equivalence of orthogonal $G$-spectra. \end{prop}

\begin{coro} \label{Fequi2} A morphism $f \colon X \lra Y$ of $G$-equivariant orthogonal spectra is an $\F$-equivalence if and only if for any $H \in \F$, the induced map 
$$\Phi^H (\Res^G_H (f)) \colon \Phi^H (\Res^G_H (X)) \lra \Phi^H (\Res^G_H (Y))$$
on $H$-geometric fixed points is a stable equivalence of (non-equivariant) spectra. \end{coro}

\begin{proof} By Proposition \ref{Fequi1}, $f \colon X \lra Y$ is an $\F$-equivalence if and only if 
$$1 \wedge f \colon E\F_+ \wedge X \lra E\F_+ \wedge Y$$ 
is a stable equivalence of orthogonal $G$-spectra. But the latter is the case if and only if
$$\Phi^H (\Res^G_H (1 \wedge f)) \colon \Phi^H (\Res^G_H (E\F_+ \wedge X)) \lra \Phi^H (\Res^G_H (E\F_+ \wedge Y))$$
is a stable equivalence of spectra for any subgroup $H \leq G$ (\cite[XVI.6.4]{alaska}). Now using that the restriction and geometric fixed points commute with smash products as well as the defining properties of $E \F$, we obtain the desired result. \end{proof}

\;

\;

By definition of $\F$-equivalences and $\F$-cofibrations we get a Quillen adjunction
$$\xymatrix{ \id  \colon \Sp_{G, \F}^O \ar@<0.5ex>[r] & \Sp^O_G \cocolon \id. \ar@<0.5ex>[l]}$$
After deriving this Quillen adjunction one obtains an adjunction 
$$\xymatrix{ \LL  \colon \Ho(\Sp_{G, \F}^O) \ar@<0.5ex>[r] & \Ho(\Sp^O_G) \cocolon \R \ar@<0.5ex>[l]}$$
on the homotopy level. We now examine the essential image of the left adjoint functor $\LL$. Since a weak equivalence in $\Sp^O_G$ is also a weak equivalence in $\Sp_{G, \F}^O$, the unit 
$$\id \lra \R \LL$$ 
of the adjunction $(\LL, \R)$ is an isomorphism of functors. Hence the functor 
$$\LL \colon \Ho(\Sp_{G, \F}^O) \lra \Ho(\Sp^O_G)$$ 
is fully faithful. 

\begin{prop} \label{Lcompu} For any $X \in \Sp_{G, \F}^O$, there are natural isomorphisms
$$\LL(X) \cong E\F_+ \wedge ^{\LL} X \cong E\F_+ \wedge X.$$ \end{prop}

\begin{proof} Let $\lambda_X \colon X^c \lra X$ be a (functorial) cofibrant replacement of $X$ in $\Sp_{G, \F}^O$. By \cite[Theorem IV.6.10]{MM02}, the projection map $E\F_+ \wedge X^c \lra X^c$ is a weak equivalence in $\Sp^O_G$. On the other hand, Proposition \ref{Fequi1} implies that the morphism of $G$-spectra $1 \wedge \lambda_X \colon E\F_+ \wedge X^c \lra E\F_+ \wedge X$ is a weak equivalence in $\Sp^O_G$. This completes the proof.  \end{proof}

Next, note that the triangulated category $\Ho(\Sp^O_{G, \F})$ is compactly generated with 
$$\{\Sigma_+^{\infty} G/H \;|\; H \in \F \}$$ 
as a set of compact generators. Indeed, this follows from the following chain of isomorphisms:
\begin{align*} \begin{split} [\Sigma_+^{\infty} G/H, X]_*^{\Ho(\Sp_{G, \F}^O)} \cong [E\F_+ \wedge \Sigma_+^{\infty} G/H, E\F_+ \wedge X]_*^G \cong \\  [\Sigma_+^{\infty} G/H, E\F_+ \wedge X]_*^G \cong \pi_*^H(E\F_+ \wedge X) \cong \pi_*^H X. \end{split} \end{align*}
The first isomorphism in this chain follows from Proposition \ref{Lcompu} and from the fact that $\LL$ is fully faithful. The second isomorphism holds since $H \in \F$. Finally, the last isomorphism is an immediate consequence of Corollary \ref{Fequi2}.
\begin{prop} \label{Limage} The essential image of the functor $ \LL \colon \Ho(\Sp_{G, \F}^O) \lra \Ho(\Sp^O_G)$ is exactly the localizing subcategory generated by $\{\Sigma_+^{\infty} G/H \; |\; H \in \F \}$. \end{prop} 

\begin{proof} The functor $\LL$ is exact and as we already noted, $\Ho(\Sp_{G, \F}^O)$ is generated by the set $\{\Sigma_+^{\infty} G/H \;|\; H \in \F \}$. Next, by Proposition \ref{Lcompu}, for any $H \in \F$, 
$$\LL(\Sigma_+^{\infty} G/H) \cong E \F_+ \wedge \Sigma_+^{\infty} G/H.$$
The projection map $E \F_+ \wedge \Sigma_+^{\infty} G/H \lra \Sigma_+^{\infty} G/H$ is a weak equivalence in $\Sp^O_G$. The rest of the proof follows from the fact that $\LL$ is full. \end{proof}

\;

\;

Next, we need the following simple lemma from category theory.

\begin{lem} Let 
$$\xymatrix{ \LL \colon \D \ar@<0.5ex>[r] & \E \cocolon \R. \ar@<0.5ex>[l]}$$
be an adjunction and assume that the unit
$$\text{id} \lra \R \LL$$
is an isomorphism (or, equivalently, $\LL$ is fully faithful). Further, suppose we are given morphisms
$$\xymatrix{X \ar[r]^-\alpha & Z & Y \ar[l]_-\beta}$$  
in $\E$ such that $X$ and $Y$ are in the essential image of $\LL$ and $\R(\alpha)$ and $\R(\beta)$ are isomorphisms in $\D$. Then there is an isomorphism $\xymatrix{\gamma \colon X \ar[r]^-{\cong} & Y}$ in $\E$ such that the diagram
$$\xymatrix{X \ar[rr]^{\gamma} \ar[dr]_\alpha & & Y \ar[ld]^\beta \\ & Z }$$
commutes. \end{lem}

\begin{proof} One has the commutative diagram
$$\xymatrix{\LL \R (X) \ar[d]^{\counit}_{\cong} \ar[r]^{\LL \R (\alpha)}_{\cong} & \LL \R (Z) \ar[d]^{\counit} & \LL \R (Y) \ar[l]_{\LL \R (\beta)}^{\cong} \ar[d]^{\counit}_{\cong} \\ X \ar[r]^\alpha & Z & Y, \ar[l]_\beta}$$
where the left and right vertical arrows are isomorphisms since $X$ and $Y$ are in the essential image of $\LL$ and the functor $\LL$ is fully faithful. We can choose $\xymatrix{\gamma \colon X \ar[r]^-{\cong} & Y}$ to be the composite
\[\xymatrix{X \ar[rr]^{\counit^{-1}} & & \LL \R (X) \ar[rr]^{\LL \R (\alpha)} & & \LL \R (Z)  \ar[rr]^{(\LL \R (\beta))^{-1}} & & \LL \R (Y) \ar[rr]^{\counit} & & Y.} \qedhere \] \end{proof}

\begin{coro} \label{zigzag} Let $\F$ be a family of subgroups of $G$ and suppose $X$ and $Y$ are in the essential image of $\LL \colon \Ho(\Sp_{G, \F}^O) \lra \Ho(\Sp^O_G)$ (which is the localizing subcategory generated by $\{\Sigma_+^{\infty} G/H \;|\; H \in \F \}$ according to \ref{Limage}). Further assume that we have maps
$$\xymatrix{X \ar[r]^-\alpha & Z & Y \ar[l]_-\beta}$$
such that $\pi_*^H \alpha$ and $\pi_*^H \beta$ are isomorphisms for any $H \in \F$ (Or, in other words, $\alpha$ and $\beta$ are $\F$-equivalences.). Then there is an isomorphism $\xymatrix{\gamma \colon X \ar[r]^-{\cong} & Y}$ such that the diagram
$$\xymatrix{X \ar[rr]^{\gamma} \ar[dr]_\alpha & & Y \ar[ld]^\beta \\ & Z }$$
commutes. \end{coro}

\begin{proof} We apply the previous lemma to the adjunction 
$$\xymatrix{ \LL  \colon \Ho(\Sp_{G, \F}^O) \ar@<0.5ex>[r] & \Ho(\Sp^O_G) \cocolon \R \ar@<0.5ex>[l]}$$
and use the isomorphism $\pi_*^H \R(T) \cong \pi_*^HT$, $H \in \F$. \end{proof}

\subsection{Inductive strategy and preservation of induced classifying spaces} \label{indstrat} Recall that we are given an exact functor of triangulated categories 
$$\xymatrix{ F \colon \Ho(\Sp_G^O,_{(2)}) \ar[r] & \Ho(\Sp_G^O,_{(2)})}$$
that preserves arbitrary coproducts and such that 
$$F(\Sigma_{+}^{\infty} G/H) = \Sigma_{+}^{\infty} G/H, \;\; H \leq G,$$ 
and
$$F(g)=g, \;\; F(\res_{K}^H)= \res_{K}^H, \;\; F(\tr_{K}^H)=\tr_{K}^H, \;\; g \in G, \;\; K \leq H \leq G.$$
We want to show that $F$ is an equivalence of categories. Proposition \ref{splinj} implies that in order to prove that $F$ is an equivalence, it suffices to check that for any subgroup $H \leq G$, the map between graded endomorphism rings
$$F \colon [\Sigma_{+}^{\infty} G/H, \Sigma_{+}^{\infty} G/H]_*^G \lra [F(\Sigma_{+}^{\infty} G/H), F(\Sigma_{+}^{\infty} G/H)]_*^G = [\Sigma_{+}^{\infty} G/H, \Sigma_{+}^{\infty} G/H]_*^G$$
is an isomorphism. The strategy is to do this inductively. We proceed by induction on the cardinality of $H$.  The induction starts with the case $H=e$. Proposition \ref{gfree} tells us that the map
$$F \colon [\Sigma_{+}^{\infty} G, \Sigma_{+}^{\infty} G]_*^G \lra [\Sigma_{+}^{\infty} G, \Sigma_{+}^{\infty} G]_*^G$$
is an isomorphism and hence the basis step is proved. The induction step follows from the next proposition which is one of the main technical results of this paper:

\begin{prop} \label{imptech} Let $G$ be a finite group and $H$ a subgroup of $G$. Assume that for any subgroup $K$ of $G$ which is proper subconjugate to $H$, the map
$$F \colon [\Sigma_+^{\infty} G/K, \Sigma_+^{\infty} G/K]_*^G \lra [F(\Sigma_+^{\infty} G/K), F(\Sigma_+^{\infty} G/K)]_*^G = [\Sigma_+^{\infty} G/K, \Sigma_+^{\infty} G/K]_*^G$$
is an isomorphism. Then the map
$$F \colon [\Sigma_+^{\infty} G/H, \Sigma_+^{\infty} G/H]_*^G \lra [F(\Sigma_{+}^{\infty} G/H), F(\Sigma_{+}^{\infty} G/H)]_*^G = [\Sigma_+^{\infty} G/H, \Sigma_+^{\infty} G/H]_*^G$$
is an isomorphism. \end{prop}

Before starting to prove this proposition, one has to show that under its assumptions the functor $F$ preserves the object $G \ltimes_H \Sigma_+^{\infty} E\P[H]$. More precisely, let $\P[H]$ denote the family of proper subgroups of $H$. This is a family with respect to $H$ and not necessarily with respect to the whole group $G$. Next, let 
$$\proj \colon G \ltimes_H \Sigma_+^{\infty} E\P[H] \lra \Sigma_+^{\infty} G/H$$ 
be the projection (as in Subsection \ref{sec1}). The following holds:

\begin{lem} \label{presclass} Suppose $G$ is a finite group and $H$ a subgroup of $G$. Assume that for any subgroup $K$ of $G$ which is proper subconjugate to $H$, the map
$$F \colon [\Sigma_+^{\infty} G/K, \Sigma_+^{\infty} G/K]_*^G \lra [\Sigma_+^{\infty} G/K, \Sigma_+^{\infty} G/K]_*^G$$
is an isomorphism. Then there is an isomorphism 
$$\xymatrix{\gamma \colon F(G \ltimes_H \Sigma_+^{\infty} E\P[H]) \ar[r]^-{\cong} & G \ltimes_H \Sigma_+^{\infty} E\P[H]}$$ 
such that the diagram 
$$\xymatrix{F(G \ltimes_H \Sigma_+^{\infty} E\P[H]) \ar[d]_{F(\proj)} \ar[r]_-{\gamma}^-{\cong} & G \ltimes_H \Sigma_+^{\infty} E\P[H] \ar[d]^{\proj} \\ F(\Sigma_+^{\infty} G/H) \ar@{=}[r] & \Sigma_+^{\infty} G/H}$$
commutes. \end{lem}

\begin{proof} Let $\P[H|G]$ denote the family of subgroups of $G$ which are proper subconjugate to $H$. By Proposition \ref{Limage}, the essential image of the fully faithful embedding
$$\LL \colon \Ho(\Sp_{G, \P[H|G]}^O) \lra \Ho(\Sp_G^O).$$
is the localizing subcategory generated by the set $\{\Sigma_+^{\infty} G/K \; | \; K \in \P[H|G] \}$. Obviously, the spectrum $G \ltimes_H \Sigma_+^{\infty} E\P[H]$ is an object of this localizing subcategory as the $H$-CW complex $E \P[H]$ is built out of $H$-cells of orbit type $H/K$ with $K \leq H$ and $K \neq H$. Next, since the endofunctor $F \colon \Ho(\Sp^O_G) \lra \Ho(\Sp^O_G)$ is exact, preserves infinite coproducts and $F(\Sigma_+^{\infty} G/L) = \Sigma_+^{\infty} G/L$ for any $L \leq G$, the spectrum $F(G \ltimes_H \Sigma_+^{\infty} E\P[H])$ is contained in the essential image of $\LL \colon \Ho(\Sp_{G, \P[H|G]}^O) \lra \Ho(\Sp_G^O)$ as well. Hence by Corollary~\ref{zigzag}, it suffices to show the maps in the zigzag
$$\xymatrix{F(G \ltimes_H \Sigma_+^{\infty} E\P[H]) \ar[r]^-{F(\proj)} & F(\Sigma_+^{\infty} G/H) = \Sigma_+^{\infty} G/H & G \ltimes_H \Sigma_+^{\infty} E\P[H] \ar[l]_-{\proj}}$$
are $\P[H|G]$-equivalences (which means that they induce isomorphisms on $\pi_*^K(-)$ for any subgroup $K \in \P[H|G]$.). It is easy to see that the map 
$$\proj \colon G \ltimes_H \Sigma_+^{\infty} E\P[H] \lra \Sigma_+^{\infty} G/H$$
is $\P[H|G]$-equivalence. Indeed, by Corollary \ref{wirth2}, for any $K \in \P[H|G]$, one has a commutative diagram
$$\xymatrix{ \pi_*^K(G \ltimes_H \Sigma_+^{\infty} E\P[H]) \ar[rrr]^-{\proj_*} \ar[d]_{\cong} & & & \pi_*^K(\Sigma_+^{\infty} G/H) \ar[d]^{\cong} \\ [\Sigma_+^{\infty} G/K, G \ltimes_H \Sigma_+^{\infty} E\P[H]]^G_* \ar[rrr]^-{\proj_*} \ar[d]_{\cong} & & & [\Sigma_+^{\infty} G/K, G \ltimes_H \S]^G_* \ar[d]^{\cong} \\ \bigoplus_{[g] \in K \setminus G /H}  [\S, \Sigma_+^{\infty} \Res_{K \cap {}^g H}^{{}^g H} (c_g^*(E \P[H]))]_*^{K \cap {}^g H} \ar[rrr]^-{\bigoplus\limits_{[g] \in K \setminus G /H}  (\proj)_*} & & & \bigoplus_{[g] \in K \setminus G /H}  [\S, \S ]_*^{K \cap {}^g H}. } $$
If $L$ is a subgroup of $K \cap {}^g H$, then $g^{-1} L g$ is a subgroup of $H$. In fact, $g^{-1} L g$ is a proper subgroup of $H$ since $K \in \P[H|G]$. This implies that for any subgroup $L \leq K \cap {}^g H$ the space $(\Res_{K \cap {}^g H}^{{}^g H} (c_g^*(E \P[H])))^L= (E\P[H])^{g^{-1}Lg}$ is contractible. Hence, $\Res_{K \cap {}^g H}^{{}^g H} (c_g^*(E \P[H]))$  is a $(K \cap {}^g H)$-contractible cofibrant $(K \cap {}^g H)$-space and we see that the map  
$$\proj \colon \Sigma_+^{\infty} \Res_{K \cap {}^g H}^{{}^g H} (c_g^*(E \P[H])) \lra \S$$ 
is a $(K \cap {}^g H)$-equivalence. From this we conclude that the lower horizontal map in the latter commutative diagram is an isomorphism. Hence, the upper horizontal map is an isomorphism for any subgroup $K \in \P[H|G]$ and one concludes that the map 
$$\proj \colon G \ltimes_H \Sigma_+^{\infty} E\P[H] \lra \Sigma_+^{\infty} G/H$$ 
is a $\P[H|G]$-equivalence. 

It remains to show that the morphism $F(\proj) \colon F(G \ltimes_H \Sigma_+^{\infty} E\P[H]) \lra F(\Sigma_+^{\infty} G/H)$ is a $\P[H|G]$-equivalence as well. We first note that the assumptions imply that for any $K \in \P[H|G]$ and any (not necessarily proper) subgroup $L \leq H$, the map
$$F \colon [\Sigma_+^{\infty} G/K, \Sigma_+^{\infty} G/L]^G_* \lra [F(\Sigma_+^{\infty} G/K), F(\Sigma_+^{\infty} G/L)]^G_* = [\Sigma_+^{\infty} G/K, \Sigma_+^{\infty} G/L]^G_*$$
is an isomorphism. Indeed, this follows from Proposition \ref{splinj} as well as from the commutative diagram
$$\xymatrix{[\Sigma_{+}^{\infty} G/K, \Sigma_{+}^{\infty} G/L]_*^G  \ar[rr]^-{(\kappa_\lb)_{[\lb] \in K \setminus G /L}} \ar[d]_-F & & \bigoplus_{[\lb] \in K \setminus G /L} [\Sigma_{+}^{\infty} G/({}^{\lb} L \cap K) , \Sigma_{+}^{\infty} G/({}^{\lb} L \cap K)]_*^G  \ar[d]^-{ \bigoplus_{[\lb] \in K \setminus G /L} F} \\ [\Sigma_{+}^{\infty} G/K, \Sigma_{+}^{\infty} G/L]_*^G  \ar[rr]^-{(\kappa_\lb)_{[\lb] \in K \setminus G /L}} &  &  \bigoplus_{[\lb] \in K \setminus G /L} [\Sigma_{+}^{\infty} G/({}^{\lb} L \cap K) , \Sigma_{+}^{\infty} G/({}^{\lb} L \cap K)]_*^G }$$
where the right vertical map is an isomorphism since ${}^{\lb} L \cap K$ is proper subconjugate to $H$ for any $\lb$. In particular, the map
$$F \colon [\Sigma_+^{\infty} G/K, \Sigma_+^{\infty} G/H]^G_* \lra [\Sigma_+^{\infty} G/K, \Sigma_+^{\infty} G/H]^G_*$$
is an isomorphism. Next, using a standard argument on triangulated categories, we see that for any $K \in \P[H|G]$ and any $X$ from the localizing subcategory of $\Ho(\Sp^O_G)$ generated by $\{\Sigma_+^{\infty} G/L \; |\; L \leq H \}$, the map
$$F \colon [\Sigma_+^{\infty} G/K, X]^G_* \lra [F(\Sigma_+^{\infty} G/K), F(X)]^G_*$$
is an isomorphism (recall $F(\Sigma_+^{\infty} G/L)=\Sigma_+^{\infty} G/L$ for any $L \leq G$). As a consequence, we see that the morphism
$$F \colon [\Sigma_+^{\infty} G/K, G \ltimes_H \Sigma_+^{\infty} E\P[H]]^G_* \lra [F(\Sigma_+^{\infty} G/K), F(G \ltimes_H \Sigma_+^{\infty} E\P[H])]^G_*$$
is an isomorphism. Finally, for any $K \in \P[H|G]$, consider the commutative diagram
$$\xymatrix{[\Sigma_+^{\infty} G/K, G \ltimes_H \Sigma_+^{\infty} E\P[H]]^G_* \ar[rr]^-{\proj_*} \ar[d]^F_{\cong} & & [\Sigma_+^{\infty} G/K, \Sigma_+^{\infty} G/H]^G_* \ar[d]^-F_-{\cong} \\ [F(\Sigma_+^{\infty} G/K), F(G \ltimes_H \Sigma_+^{\infty} E\P[H])]^G_* \ar[rr]^-{F(\proj)_*} \ar@{=}[d] & & [F(\Sigma_+^{\infty} G/K), F(\Sigma_+^{\infty} G/H )]^G_* \ar@{=}[d] \\  [\Sigma_+^{\infty} G/K, F(G \ltimes_H \Sigma_+^{\infty} E\P[H])]^G_* \ar[rr]^-{F(\proj)_*} & & [\Sigma_+^{\infty} G/K, F(\Sigma_+^{\infty} G/H )]^G_*. }$$
As we already explained, the upper horizontal map is an isomorphism. Thus the lower horizontal map in this diagram is an isomorphism as well and therefore, the map 
$$F(\proj) \colon F(G \ltimes_H \Sigma_+^{\infty} E\P[H]) \lra F(\Sigma_+^{\infty} G/H)$$ 
is a $\P[H|G]$-equivalence. \end{proof} 

\subsection{Completing the proof of Theorem \ref{reduction}} In this subsection we continue the induction started in the previous subsection and prove Proposition \ref{imptech}. Finally, at the end, we complete the proof of Theorem \ref{reduction} and hence prove the main Theorem \ref{damtkicebuli}. 

\bigskip

\begin{proof_ohne_pt} {\bf of Proposition \ref{imptech}.} Recall (Section \ref{geofixsec}) that the extension
$$\xymatrix{1 \ar[r] & H \ar[r]^-{\iota} & N(H) \ar[r]^-{\varepsilon} & W(H) \ar[r] & 1.}$$
determines the inflation functor
$$\varepsilon^* \colon \Ho(\Sp^O_{W(H)}) \lra \Ho(\Sp^O_{N(H)})$$
and the geometric fixed point functor
$$\Phi^H \colon \Ho(\Sp^O_{N(H)}) \lra \Ho(\Sp^O_{W(H)}).$$
Let $\hat{F} \colon \Ho(\Sp^O_{W(H)}) \lra \Ho(\Sp^O_{W(H)})$ denote the composite
{\footnotesize $$\xymatrix{\Ho(\Sp^O_{W(H)}) \ar[r]^-{\varepsilon^*} & \Ho(\Sp^O_{N(H)}) \ar[r]^-{G \ltimes_{N(H)} -} & \Ho(\Sp_G^O) \ar[r]^F & \Ho(\Sp^O_G)  \ar[r]^-{\Res_{N(H)}^G} & \Ho(\Sp_{N(H)}^O) \ar[r]^{\Phi^H} & \Ho(\Sp^O_{W(H)}).}$$}
\hspace{-0.3cm} It follows from the identifications we did in Subsection \ref{Weyl1} and from the properties of $F$ that the functor $\hat{F}$ is exact, preserves infinite coproducts and sends $\Sigma_+^{\infty} W(H)$ to itself. Moreover, it also follows that the restriction
$$\hat{F}|_{\Ho(\Mod \Sigma_+^{\infty} W(H))} \colon \Ho(\Mod \Sigma_+^{\infty} W(H)) \lra \Ho(\Mod \Sigma_+^{\infty} W(H))$$
of $\hat{F}$ on the localizing subcategory of $\Ho(\Sp^O_{W(H)})$ generated by $\Sigma_+^{\infty} W(H)$ satisfies the assumptions of Proposition \ref{gfree}. Hence, the map
$$\hat{F} \colon [\Sigma_+^{\infty} W(H), \Sigma_+^{\infty} W(H)]^{W(H)}_* \lra  [\Sigma_+^{\infty} W(H), \Sigma_+^{\infty} W(H)]^{W(H)}_*$$
is an isomorphism. Next, by the assumptions and Proposition \ref{splinj} (like in the proof of Lemma \ref{presclass}), we see that for any proper subgroup $L$ of $H$, the map
$$F \colon [\Sigma_+^{\infty} G/H, \Sigma_+^{\infty} G/L]^G_* \lra [\Sigma_+^{\infty} G/H, \Sigma_+^{\infty} G/L]^G_*$$
is an isomorphism. This, using a standard argument on triangulated categories, implies that for any $X$ which is contained in the localizing subcategory of $\Ho(\Sp^O_G)$ generated by $\{\Sigma_+^{\infty} G/L \;|\; L \in  \P[H] \}$, the map
$$F \colon [\Sigma_+^{\infty} G/H, X]^G_* \lra [F(\Sigma_+^{\infty} G/H), F(X)]^G_*$$
is an isomorphism and hence, in particular, so is the morphism
$$F \colon [\Sigma_+^{\infty} G/H, G \ltimes_H \Sigma_+^{\infty} E\P[H]]^G_* \lra [F(\Sigma_+^{\infty} G/H), F(G \ltimes_H \Sigma_+^{\infty} E\P[H])]^G_*.$$

Finally, we have the following important commutative diagram
{\footnotesize $$\hspace{-0.8cm}\xymatrix{[\Sigma_+^{\infty} G/H, G \ltimes_H \Sigma_+^{\infty} E\P[H]]_*^G \ar[r]^-{\proj_*} \ar[d]^-F_{\cong} & [\Sigma_+^{\infty} G/H, \Sigma_+^{\infty} G/H]^G_* \ar[d]^F & [\Sigma_+^{\infty} W(H), \Sigma_+^{\infty} W(H)]^{W(H)}_* \ar[l]_-{G\ltimes_{N(H)} \varepsilon^*} \ar[d]^-{\hat{F}}_-{\cong}\\  [F(\Sigma_+^{\infty} G/H), F(G \ltimes_H \Sigma_+^{\infty} E\P[H])]_*^G \ar[r]^-{F(\proj)_*} \ar[d]_-{\cong} & [F(\Sigma_+^{\infty} G/H), F(\Sigma_+^{\infty} G/H)]^G_* \ar[r]^-{\Phi^H} \ar@{=}[d] & [\hat{F}(\Sigma_+^{\infty} W(H)), \hat{F}(\Sigma_+^{\infty} W(H))]^{W(H)}_* \ar@{=}[d] \\ [\Sigma_+^{\infty} G/H, G \ltimes_H \Sigma_+^{\infty} E\P[H]]_*^G \; \ar@{>->}[r]^-{\proj_*} & [\Sigma_+^{\infty} G/H, \Sigma_+^{\infty} G/H]^G_* \ar@{->>}[r]^-{\Phi^H} & [\Sigma_+^{\infty} W(H), \Sigma_+^{\infty} W(H)]^{W(H)}_*.}$$}
\hspace{-0.38cm} Lemma \ref{presclass} implies that the lower left square commutes and the lower left vertical map is an isomorphism. Other squares commute by definitions. Further, according to Proposition \ref{ses}, the lower row in this diagram is a short exact sequence and hence so is the middle one.

Now a simple diagram chase shows that the map
$$F \colon [\Sigma_+^{\infty} G/H, \Sigma_+^{\infty} G/H]^G_* \lra [F(\Sigma_+^{\infty} G/H), F(\Sigma_+^{\infty} G/H)]^G_* =[\Sigma_+^{\infty} G/H, \Sigma_+^{\infty} G/H]^G_*$$
is an isomorphism. Indeed, assume that $\ast > 0$ (the case $\ast=0$ is obvious by the assumptions on $F$). Then the latter map has the same finite source and target and hence it suffices to show that it is surjective. Fix $\ast > 0$ and take any $\alpha \in [F(\Sigma_+^{\infty} G/H), F(\Sigma_+^{\infty} G/H)]^G_*$. Since the map 
$$\hat{F} \colon  [\Sigma_+^{\infty} W(H), \Sigma_+^{\infty} W(H)]^{W(H)}_* \lra [\hat{F}(\Sigma_+^{\infty} W(H)), \hat{F}(\Sigma_+^{\infty} W(H))]^{W(H)}_*$$ 
is an isomorphism, there exists $\beta \in [\Sigma_+^{\infty} W(H), \Sigma_+^{\infty} W(H)]^{W(H)}_*$ such that
$$\hat{F}(\beta) = \Phi^H(\alpha).$$
By definition of the functor $\hat{F}$, the element 
$$F(G \ltimes_{N(H)} \varepsilon^*(\beta)) - \alpha \in [F(\Sigma_+^{\infty} G/H), F(\Sigma_+^{\infty} G/H)]^G_*$$ 
is in the kernel of
$$\Phi^H \colon [F(\Sigma_+^{\infty} G/H), F(\Sigma_+^{\infty} G/H)]^G_* \lra [\hat{F}(\Sigma_+^{\infty} W(H)), \hat{F}(\Sigma_+^{\infty} W(H))]^{W(H)}_*. $$
But the kernel of this map is contained in the image of 
$$F \colon [\Sigma_+^{\infty} G/H, \Sigma_+^{\infty} G/H]^G_* \lra [F(\Sigma_+^{\infty} G/H), F(\Sigma_+^{\infty} G/H)]^G_*$$
since the middle row in the commutative diagram above is exact and the upper left vertical map is an isomorphism. Consequently, $F(G \ltimes_{N(H)} \varepsilon^*(\beta)) - \alpha$ is in the image of $F$ and this completes the proof. \end{proof_ohne_pt}

\begin{proof_ohne_pt} {\bf of Theorem \ref{reduction}.} Now we continue with the induction. Recall, that our aim is to show that for any subgroup $H \in G$, the map 
$$F \colon [\Sigma_+^{\infty} G/H, \Sigma_+^{\infty} G/H]^G_* \lra [\Sigma_+^{\infty} G/H, \Sigma_+^{\infty} G/H]^G_*$$
is an isomorphism. The strategy that was indicated at the beginning of Subsection \ref{indstrat} is to proceed by induction on the cardinality of $H$. The induction basis follows from Proposition \ref{gfree} as we already explained. Now suppose $n > 1$, and assume that the claim holds for all subgroups of $G$ with cardinality less than or equal to $n-1$. Let $H$ be any subgroup of $G$ that has cardinality equal to $n$. Then, by the induction assumption, for any subgroup $K$ which is proper subconjugate to $H$, the map
$$F \colon [\Sigma_+^{\infty} G/K, \Sigma_+^{\infty} G/K]^G_* \lra [\Sigma_+^{\infty} G/K, \Sigma_+^{\infty} G/K]^G_*$$
is an isomorphism. Proposition \ref{imptech} now implies that 
$$F \colon [\Sigma_+^{\infty} G/H, \Sigma_+^{\infty} G/H]^G_* \lra [\Sigma_+^{\infty} G/H, \Sigma_+^{\infty} G/H]^G_*$$
is an isomorphism and this completes the proof of the claim.

The rest follows from Proposition \ref{splinj} as already explained in Section \ref{redtrres}. \end{proof_ohne_pt}

\vspace{0.5cm}

\noindent Department of Mathematical Sciences \\ University of Copenhagen \\
Universitetsparken 5 \\ 2100 Copenhagen $\emptyset$\\ Denmark

\;

\;

\noindent E-mail address: \texttt{irakli.p@math.ku.dk}

\end{document}